\documentclass{amsart}
\usepackage[table]{xcolor} 
\usepackage[colorlinks,citecolor=darkblue,urlcolor=blue,bookmarks=false,hypertexnames=true]{hyperref} 
\usepackage{amsrefs}
\usepackage{amsmath}
\usepackage{amsfonts}
\usepackage{amssymb}
\newtheorem{theo}{Theorem}[section]
\newtheorem{prop}{Proposition}[section]
\newtheorem{coro}{Corollary}[section]
\newtheorem{lem}{ Lemma}[section]
\newtheorem{definition}{Definition}
\newtheorem{propdef}{Proposition-Definition}
\newtheorem{step}{Step}[section]
\definecolor{darkblue}{HTML}{004C93} 

\def \rr {\mathbb{R}}
\def \crit{2^{\star}(s)}
\def \na{\mathbb{\mu_{\alpha}}}
\def \vv {\mathbb{\Vert}}	
\def \rn {\mathbb{R}^n}
\def \sn {\mathbb{S}^{n-1}}
\def \nn {\mathbb{N}}
\def \ll {\mathbb{\lambda}}
\def \ep {\epsilon}
\def \bb {\hbox}
\def \ma {\mu_{\alpha}}
\def \uta  {\tilde{u}_{\alpha}}
\def \ua  {u_{\alpha}}
\def \gta {\tilde{g}_{\alpha}}
\def \td {\theta_{R}}
\def \uha  {\hat{u}_{\alpha}}
\def \nua {\nu_{\alpha}}
\def \utza  {\tilde{u}_{\alpha}}
\def \ut  {\tilde{u}}
\date{June 23th, 2020}

\title[The second best constant for the Hardy-Sobolev inequality]{The second best constant for the Hardy-Sobolev inequality on manifolds}

\author{Hussein Cheikh Ali}
\address{D\'epartement de Math\'ematiques, Universit\'e libre de Bruxelles, CP 214, Boulevard du Triomphe, B-1050 Bruxelles, Belgium}

\email{Hussein.cheikh-ali@ulb.ac.be}
\begin{document}
	
\begin{abstract} We consider the second best constant in the Hardy-Sobolev inequality on a Riemannian manifold. More precisely, we are interested with the existence of extremal functions for this inequality. This problem was tackled by Djadli-Druet \cite{DD} for Sobolev inequalities. Here, we establish the corresponding result for the singular case. In addition, we perform a blow-up analysis of solutions Hardy-Sobolev equations of minimizing type. This yields informations on the value of the second best constant in the related Riemannian functional inequality. 
		\end{abstract}
		\maketitle

	\tableofcontents

		\section{Introduction}
	Let $(M,g)$ be a compact Riemannian manifold of dimension $n\geq3$ without boundary, $d_g$ be the Riemannian distance on $M$ and $H_1^2(M)$ be the completion of $C^\infty(M)$ for the norm $u\mapsto \Vert u\Vert_2+\Vert\nabla u\Vert_2$. We fix $x_0\in M$, $s\in \left[0,2\right)$ and we let $\crit:= \frac{2(n-s)}{n-2}$ be the critical Hardy-Sobolev exponent. Here, and in the sequel, we set
	$$u\to \vv u \vv_{p,s}:= \left( \int_M |u|^{p}d_g(\cdot,x_0)^{-s}\, dv_g\right)^{\frac{1}{p}}$$ and we define $L^p\left(M, d_g(x,x_0)^{-s}\right)=\{u\in L^1(M)/\, \Vert u\Vert_{p,s}<\infty\}$ where $dv_g$ is the Riemannian element of volume. The  Hardy-Sobolev embedding theorem $H_1^2(M)\hookrightarrow L^{\crit}\left(M, d_g(x,x_0)^{-s}\right)$ yields $A,B>0$ such that
\begin{equation}\label{ineq:AB}
\vv u \vv_{\crit,s}^2\leq A\vv \nabla u \vv_{2}^2+B\vv  u \vv_{2}^2
\end{equation}
for all $u\in H_1^2(M)$.  Discussions on the Hardy and Hardy-Sobolev inequalities are in Ghoussoub-Moradifam \cite{GM}.  When $s=0$, this is the classical Sobolev inequality, and extensive discussions on the optimal values of the constants are in the monograph Druet-Hebey \cite{DH:AB}. It was proved by Hebey-Vaugon \cite{HV} (the classical case $s=0$) and by Jaber \cite{J2} ($s\in (0,2)$) that 
$$\mu_s(\rr^n)^{-1}=\inf\{A>0\hbox{ s.t. } \bb{there exists }B>0\hbox{ for which } \eqref{ineq:AB} \bb{ is true}\},$$
and that the infimum is achieved, where
$$\mu_s(\rr^n)=\inf\left\lbrace\frac{\int_{\rr^n}|\nabla u|^2\,dX}{\left(\int_{\rr^n}\frac{|u|^{\crit}}{|X|^{s}}\,dX \right)^{\frac{2}{\crit}}}, u\in C^{\infty}_{c}(\rr^n)\right\rbrace $$	
is the best constant in the Hardy-Sobolev inequality (see Lieb \cite{lieb} Theorem 4.3 for the value). Therefore, there exists $B>0$ such that
\begin{equation}\label{ineq:B}
\vv u \vv_{\crit,s}^2\leq \mu_s(\rr^n)^{-1}\left(\vv \nabla u \vv_{2}^2+B\vv  u \vv_{2}^2\right)
\end{equation}
for all $u\in H_1^2(M)$. Saturating this inequality with repect to $B$, we define the second best constant as
$$B_s(g):=\inf\{B>0\hbox{ s.t. }\eqref{ineq:B}\hbox{ holds for all }u\in H_1^2(M)\},$$
to get the optimal inequality
\begin{equation}\label{ineq:opt}
\vv u \vv_{\crit,s}^2\leq \mu_s(\rr^n)^{-1}\left(\vv \nabla u \vv_{2}^2+B_s(g)\vv  u \vv_{2}^2\right)
\end{equation}
for all $u\in H_1^2(M)$. A remark is that it follows from the analysis of Jaber \cite{J1} that
\begin{equation*}
\left\{\begin{array}{ll}
B_{s}(g)\geq  c_{n,s}Scal_g(x_0)  &\bb{ if } n\geq 4;\\\\
\hbox{The mass of }\Delta_g+B_s(g)\hbox{ is nonpositive}& \bb{ if } n=3 ,
\end{array}\right.
\end{equation*}
where $Scal_g(x_0)$ is the scalar curvature at $x_0$ and
\begin{equation}\label{eq:cns}
c_{n,s}:=\frac{(n-2)(6-s)}{12(2n-2-s)},
\end{equation}
and the mass  will be defined in Proposition-Definition \ref{def:mass:beta}.

\smallskip\noindent In this paper, we are interested in the value of the second best constant and the existence of extremal functions for the inequality \eqref{ineq:opt}:
\begin{definition}\label{def:extremal}
We say that $u_0\in H_1^2(M)$,  $u_0\not\equiv 0$ is an extremal for \eqref{ineq:opt} if 
\begin{equation*}
\vv u_0 \vv_{\crit,s}^2= \mu_s(\rr^n)^{-1}\left(\vv \nabla u_0 \vv_{2}^2+B_s(g)\vv  u_0 \vv_{2}^2\right).
\end{equation*}
\end{definition}
\smallskip \noindent When $s=0$, the existence of extremals has been studied by Druet and al.:

\begin{theo}[The case $s=0$, \cites{DD,D}] Let $(M,g)$ be a compact Riemannian manifold of dimension $n\geq 3$. Assume that $s=0$ and that there is no extremal for \eqref{ineq:opt}. Then
\begin{itemize}
\item $B_{0}(g)=c_{n,0}\max_M \hbox{Scal}_g(x_0)$ if $n\geq 4$;
\item The mass of $\Delta_g+B_{0}(g)$ vanishes if $n=3$,
\end{itemize}
where $c_{n,0}$ is defined in \eqref{eq:cns}.
\end{theo}
We establish the corresponding  result for the singular case $s\in (0,2)$:
\begin{theo}[The case $s>0$]\label{theo:B_s(g)} Let $(M,g)$ be a compact Riemannian manifold of dimension $n\geq 3$. We fix $x_0\in M$ and $s\in (0,2)$. We assume that there is no extremal for \eqref{ineq:opt}. Then
\begin{itemize}
\item $B_s(g)=c_{n,s} \hbox{Scal}_g(x_0)$ if $n\geq 4$;
\item The mass of $\Delta_g+B_s(g)$ vanishes if $n=3$,
\end{itemize}
where $c_{n,s}$ is defined in \eqref{eq:cns}.
\end{theo}
\medskip\noindent Our proof relies on the blow-up analysis of critical elliptic equations in the spirit of Druet-Hebey-Robert \cite{DHR}. Let $(a_{\alpha})_{\alpha\in\nn}\in C^1(M)$ be such that
\begin{eqnarray}\label{m0}
	\lim_{\alpha\to +\infty} a_{\alpha}= a_{\infty} \bb{ in } C^1(M).
\end{eqnarray}
We assume uniform coercivity, that is there exists $c_0>0$ such that 
\begin{equation}\label{m0bis}
	\int_M\left( |\nabla w|_g^2 +a_{\alpha}w^2\right)\  dv_g \geq c_0 \int_M w^2 \,dv_g \bb{ for all } w \in  H^2_1(M).
\end{equation}
Note that this is equivalent to the coercivity of $\Delta_g+a_{\infty}$. We consider $(\ll_{\alpha})_{\alpha}\in (0,+\infty)$ such that 
\begin{eqnarray}\label{m1}
	\lim_{\alpha\to +\infty} \ll_{\alpha}=\mu_s(\rr^n).
\end{eqnarray}
We let  $(u_{\alpha})_{\alpha}\in H^2_1(M)$ is a sequence of weak solutions to 
\begin{equation}\label{m2}
\left\{\begin{array}{ll}
\Delta_g  u_{\alpha} + a_{\alpha}   u_{\alpha}=\lambda_{\alpha}\frac{ u_{\alpha}^{\crit-1}}{d_g(x,x_0)^s}&\hbox{ in } M, \\
u_\alpha\geq 0  &\hbox{ a.e. in } M,
\end{array}\right.
\end{equation}
where $\Delta_g:=-div_g(\nabla)$ is the Laplace-Beltrami operator. We assume that
\begin{equation}\label{m7}
	\vv u_{\alpha}\vv_{\crit,s}=1,
\end{equation}
and that
\begin{eqnarray}\label{m3}
	u_{\alpha}\rightharpoonup 0 \bb{ as } \alpha \to +\infty \bb{ weakly in  } H^2_1(M).
\end{eqnarray}
It follows from the regularity and the maximum principle of Jaber \cite{J1} that $u_{\alpha}\in C^{0,\beta_1}(M)\cap C_{loc}^{2,\beta_2}(M\backslash\{x_0\})$, $\beta_1\in (0,\min(1,2-s))$, $\beta_2\in (0,1)$ and $u_{\alpha}>0$. Therefore, since $M$ is compact, there exists $x_{\alpha}\in M$ and $\na>0$ such that 
\begin{equation}\label{m00}
	\na:=\left( \max_{M}u_{\alpha}\right) ^{-\frac{2}{n-2}}=\left( 
	u_{\alpha}(x_{\alpha})\right) ^{-\frac{2}{n-2}}.
\end{equation}
We prove two descriptions of the asymptotics of $(u_\alpha)$:
\begin{theo}\label{prop1}	Let $M$ be a compact Riemannian manifold of dimension $n\geq 3$. We fix $x_0\in M$ and $s\in (0,2)$. Let $(a_{\alpha})_{\alpha\in\nn}\in C^1(M)$ and $a_\infty\in C^1(M)$ be such that \eqref{m0} holds and  $\Delta_g+a_\infty$ is coercive in $M$. In addition, we suppose that $(\ll_{\alpha})_\alpha\in\rr$ and $(u_{\alpha})_\alpha\in H^2_1(M)$ be such that \eqref{m0} to \eqref{m00} hold for all $\alpha\in\nn$. Then, there exists $C>0$ such that,
	\begin{equation}\label{ineq:up}
	u_{\alpha}(x)\leq C\,  \frac{\ma^{\frac{n-2}{2}}}{\ma^{n-2}+d_g(x,x_0)^{n-2}}\bb{ for all } x\in M,
	\end{equation}
	where $\ma \to 0$ as $\alpha\to +\infty$ is as in \eqref{m00}.
\end{theo}

\begin{theo}\label{th1} Let $M$ be a compact Riemannian manifold of dimension $n\geq 3$. We fix $x_0\in M$ and $s\in (0,2)$. Let $(a_{\alpha})_{\alpha\in\nn}\in C^1(M)$ and $a_\infty\in C^1(M)$ be such that \eqref{m0} holds and  $\Delta_g+a_\infty$ is coercive in $M$. In addition, we suppose that $(\ll_{\alpha})_\alpha\in\rr$ and $(u_{\alpha})_\alpha\in H^2_1(M)$ be such that \eqref{m0} to \eqref{m00} hold for all $\alpha\in\nn$. Then, 
\begin{enumerate}
	\item If $n\geq 4$, then $a_{\infty}(x_0)=c_{n,s} Scal_g(x_0)$.
	\item If $n=3$, then $m_{a_{\infty}}(x_0)=0$,
\end{enumerate}
where $m_{a_{\infty}}(x_0)$ is the mass of the operator $\Delta_g+a_{\infty}$ (see Proposition-Definition \ref{def:mass:beta}) and $c_{n,s}$ is defined in \eqref{eq:cns}.
\end{theo}
The mass is defined as follows:

\begin{propdef}\label{def:mass:beta}[The mass] Let $(M,g)$ be a compact Riemannian manifold of dimension $n=3$, and let $h\in C^0(M)$ be such that $\Delta_g+h$ is coercive. Let $G_{x_0}$ be the Green's function of $\Delta_g+h$ at $x_0$. Let $\eta\in C^{\infty}(M)$ such that $\eta =1$ around $x_0$. Then there exists $\beta_{x_0}\in H_1^2(M)$ such that 
	\begin{equation}\label{green0}
	G_{x_0}=\frac{1}{4\pi}\eta d_g(\cdot,x_0)^{-1} +\beta_{x_0}\hbox{ in }M\setminus \{x_0\}.
	\end{equation}
Moreover, we have that $\beta_{x_0}\in H_2^p (M)\cap C^{0,\theta}(M)\cap C^{2,\gamma}(M\backslash\{x_0\})$ for all $p\in \left(\frac{3}{2}, 3\right)$ and $\theta,\gamma\in (0,1)$. We define the \underline{mass at $x_0$} as $m_{h}(x_0):=\beta_{x_0}(x_0)$, which is independent of the choice of $\eta$. 
\end{propdef}

\medskip\noindent Theorem \ref{th1} yields a necessary condition for the existence of solutions to \eqref{m2} that blow-up with minimal energy. Conversely, in a work in progress \cite{HCA4}, we show that this is a necessary condition by constructing an example via the finite-dimensional reduction in the spirit of Micheletti-Pistoia-V\'etois \cite{mpv}.

\medskip\noindent The role of the scalar curvature in blow-up analysis has been outlined since the reference paper \cite{druetjdg} of Druet for $s=0$. In the singular Hardy-Sobolev case ($s\in (0,2)$), the critical threshold $c_{n,s} Scal_g(x_0)$ was first observed by Jaber \cite{J1} who proved that there is a solution $u\in H_1^2(M)\cap C^0(M)$ to 
$$\Delta_gu+h u=\frac{u^{\crit-1}}{d_g(x,x_0)^s}\, ;\, u>0\hbox{ in }M.$$
when $n\geq 4$ as soon as $h(x_0)<c_{n,s} Scal_g(x_0)$ where $h\in C^0(M)$ and $\Delta_g+h$ is coercive. More recently, it was proved by Chen \cite{Chen} that for any potential $h\in C^1(M)$ such that $\Delta_g+h$ is coercive, then there is a blowing-up family of solutions $(u_\epsilon)_{\epsilon>0}$ to 
$$\Delta_gu_\epsilon+h u_\epsilon=\frac{u_\epsilon^{\crit-1-\epsilon}}{d_g(x,x_0)^s}\, ;\, u_\epsilon>0\hbox{ in }M.$$
when $h(x_0)>c_{n,s} Scal_g(x_0)$ and $n\geq 4$.

\smallskip  This paper is organized as follows. In Section \ref{sec:preliminary} we introduce some preliminary results that will be of use in the sequel. In Section \ref{sec:estimate}, We establish sharp pointwise estimates for arbitrary sequences of solutions of \ref{m2}, in particular we prove the Theorem \ref{prop1}. Section \ref{sec:c0theory} describes the $C^0$-theory for blowing-up sequences of solutions of \eqref{m2} developed in \cite{DHR}. The proof of the main Theorems \ref{theo:B_s(g)}  and \ref{th1} will be given in Sections \ref{sec:pohozaev} and \ref{sec:theo:B_s(g)}.
	\section{Preliminary blow-up analysis}\label{sec:preliminary}
We let $(a_\alpha)_{\alpha}, a_\infty\in C^1(M)$, $(\ll_\alpha)_\alpha\in\rr$ and $\ua \in H_1^2(M)$ be such that \eqref{m0}-\eqref{m00} hold. In the sequel, for any $\rho >0$ and $z\in M$ (resp. $z\in\rn$), $B_\rho(z)\subset M$ (resp. $\subset\rn$) denotes the geodesic ball of center $z$ and of radius $\rho$ in $M$ for the Riemannian distance $d_g$ (resp. in $\rn$ for the Euclidean distance).
\begin{lem}\label{co2}
We claim that
	\begin{equation*}
	\lim_{\alpha\to +\infty} \ua=0 \bb{ in } C_{loc}^0\left( M\backslash \{x_0\}\right) .
	\end{equation*}
	\end{lem}
\noindent{\it Proof of Lemma \ref{co2}:} We take $y\in M\backslash \{x_0\}$, $r_y=\frac{1}{3}d_g(y,x_0)$. Since $u_{\alpha}$ verifies the equation \eqref{m1}, we have $$\Delta_g u_{\alpha} =  H_{\alpha} u_{\alpha} \bb{ in } B_{2r_y}(y),$$ where the function 
$$ H_{\alpha}(x):= a_{\alpha}+\lambda_{\alpha}\frac{ u_{\alpha}^{\crit-2}}{d_g(x,x_0)^s}.$$
Since $a_{\alpha}\to a_{\infty}$ in $C^1$, for any $r\in (\frac{n}{2},\frac{n}{2-s})$, then there exists $c_0>0$ independant of $\alpha$ such that 
\begin{eqnarray*}
	\int_{B_{2r_y}(y)} H_{\alpha}^r \, dv_g &\leq &c_0.
\end{eqnarray*}
Using Theorem 8.11 in Gilbarg-Trudinger \cite{GT}, that there exists $C_{n,s,y,c_0}>0$ independant of $\alpha$ such that 
$$ \max_{B_{r_y}(y)} u_{\alpha}\leq C_{n,s,y,c_0} \, \vv u_{\alpha} \vv_{L^2(B_{2r_y}(y))}.$$
Therefore, it follows from the convergence in $\eqref{m3}$ that 
$$\vv u_{\alpha}\vv_{L^{\infty}(B_{r_y}(y))}\to 0 \bb{ as } \alpha \to +\infty.$$
A covering argument yields Lemma \ref{co2}. \qed
\begin{lem}\label{co3}
We claim that
\begin{equation}\label{m4}
\sup_{x\in M}u_{\alpha}(x)= +\infty \bb{ as } \alpha\to + \infty.
\end{equation}
\end{lem}
\noindent{\it Proof of Lemma \ref{co3}:} If \eqref{m4} does not hold, then there exists $C>0$ such that $$u_{\alpha}\leq C \bb{ for all } x\in M.$$
The convergence \eqref{m3} and Lebesgue's Convergence Theorem yield $\lim_{\alpha\to +\infty}\vv u_{\alpha} \vv_{\crit,s}=0$, contradiction \eqref{m7}. This proves Lemma \ref{co3}. \qed

\smallskip \noindent From the introduction (see \eqref{m00}), we recall the definition of $x_\alpha\in M$ and $\na>0$:
\begin{equation*}\label{m000}
\na:=\left( \max_{M}u_{\alpha}\right) ^{-\frac{2}{n-2}}=\left( 
u_{\alpha}(x_{\alpha})\right) ^{-\frac{2}{n-2}}.
\end{equation*}

 \medskip\noindent It  follows from Lemmae \ref{co2} and \ref{co3}  that 
 \begin{eqnarray}\label{m32}
 x_{\alpha} \to x_0 \bb{ as } \alpha \to +\infty.
 \end{eqnarray}
	
\medskip\noindent We divide the proof of  Theorem \ref{prop1} in several steps:

	\begin{step}\label{step0}
	We claim that 
	\begin{eqnarray*}
		d_g(x_{\alpha},x_0)=o(\ma) \bb{ as } \alpha \to +\infty.
	\end{eqnarray*}
\end{step}

\noindent{\it Proof of Step \ref{step0}:} With the convergence in \eqref{m32} and taking $z_{\alpha}=x_{\alpha}$ in Theorem \ref{theo1}, we get that $d_g(x_\alpha,x_0)=O(\ma)$ as $\alpha\to +\infty$. We define the rescaled metric $\bar{g}_{\alpha}(x):=\left( \exp^\star_{x_\alpha} g\right) (\mu_{\alpha}X)$ in $B_{\delta_0^{-1}\ma}(0)$ and
\begin{equation*}
\bar{u}_{\alpha}(X):=\mu_{\alpha}^{\frac{n-2}{2}} u_{\alpha}(\exp_{x_\alpha}(\mu_{\alpha}X))\bb{ for all } X\in B_{\delta_0\mu_{\alpha}^{-1}}(0)\subset \rr^n.
\end{equation*}
Here, $\exp_{x_\alpha}:B_{\delta_0}(0)\to B_{\delta_0}(x_{0})\subset M$ is the exponential map at $x_\alpha$. It follows from Theorem \ref{theo1} that
\begin{eqnarray*}
	 \bar{u}_{\alpha} \to \tilde{u} \bb{ in }  C^0_{loc}(\rr^n) \bb{ as } \alpha\to +\infty,
\end{eqnarray*}
where $\tilde{u}$ is as in Theorem \ref{theo1}. Since $\bar{u}_{\alpha}(0)=1=\max \bar{u}_\alpha$, we get 
$$ \tilde{u}(0)=\lim_{\alpha\to +\infty} \bar{u}_{\alpha}(0)=1.$$
On the other hand, we have $\vv \bar{u}_{\alpha} \vv_{\infty}=1$ thus $0$ is a maximum of $\tilde{u}$.  Let us  define $X_{0,\alpha}:= \ma^{-1} \exp_{x_\alpha}^{-1}(x_0)$ such that $X_0:= \lim_{\alpha\to +\infty} X_{0,\alpha}$. Using the explicit form of $\tilde{u}$ in Theorem \ref{theo1} that $ \tilde{u}(X)\leq \tilde{u}(X_0)$ for all $X\in \rr^n$. This yields $X_0=0$. We have that 
$$ d_g(x_{\alpha},x_0)=\ma d_{\tilde{g}_{\alpha}}(X_{0,\alpha},0)=\ma|X_{0,\alpha}|=o\left(\ma \right).$$
This yields Step \ref{step0}.
\qed

\medskip\noindent We fix $\delta_0 \in (0, i_g(M))$ where $i_g(M)>0$ is the injectivity radius of $(M,g)$. We define the metric
\begin{equation}\label{def:tga}
\tilde{g}_{\alpha}(x):=\left( \exp^\star_{x_{0}} g\right) (\mu_{\alpha}X)\hbox{ in }B_{\delta_0^{-1}\ma}(0),
\end{equation}
and the rescaled function
\begin{equation}\label{def:uta}
\tilde{u}_{\alpha}(X):=\mu_{\alpha}^{\frac{n-2}{2}} u_{\alpha}(\exp_{x_{0}}(\mu_{\alpha}X))\bb{ for all } X\in B_{\delta_0\mu_{\alpha}^{-1}}(0)\subset \rr^n,
\end{equation}
where $\exp_{x_0}$  is the exponential map at $x_0$.
Equation \eqref{m2} rewrites
\begin{equation}\label{eq:utza}
\Delta_{\tilde{g}_{\alpha}}\utza+\tilde{a}_{\alpha} \utza=\lambda_{\alpha}\frac{ \utza^{\crit-1}}{|X|^s} \bb{ in }B_{\delta_0\ma^{-1}}(0)\setminus\{0\},
\end{equation} 
where $\tilde{a}_{\alpha}(X):=\ma^2 a_{\alpha}(\exp_{x_0}(\ma X))\to 0 $ in $C^1_{loc}(\rn)$ as $\alpha\to +\infty$. 
\begin{step}\label{step00}
	We claim that,
	\begin{equation}\label{convergence}
	\lim_{\alpha\to +\infty} \tilde{u}_{\alpha}= \tilde{u}, 
	\end{equation}
	in  $C^2_{loc}(\rr^n\backslash\{0\})$  and uniformly in $C^{0,\beta}_{loc}(\rr^n)$, for all $\beta \in \left( 0,\min\{1,2-s\}\right)$.
	Where  
	\begin{equation}\label{eq:utilde}
	\tilde{u}(X)= \left( \frac{K^{2-s}}{K^{2-s}+|X|^{2-s}}\right) ^{\frac{n-2}{2-s}} \bb{ for all } X\in \rr^n\setminus\{0\},
	\end{equation} 
	with 
\begin{equation}\label{defK}
K^{2-s}=(n-2)(n-s)\mu_s(\rr^n)^{-1}.
\end{equation}
In particular, $\tilde{u}$ verifies 
	\begin{eqnarray}\label{eq:tu}
	\Delta_{Eucl} \tilde{u}= \mu_s(\rn)\frac{\tilde{u}^{\crit-1}}{|X|^s} \bb{ in } \rr^n\setminus\{0\} \bb{ and } \int_{\rr^n}\frac{\tilde{u}^{\crit}}{|X|^s} \, dX=1,
	\end{eqnarray}
	where $Eucl$ is the Euclidean metric of $\rr^n$. Moreover, 
	\begin{equation}\label{m8}
	\lim_{R\to +\infty}\lim_{\alpha\to +\infty} \int_{M\backslash B_{R\ma}(x_0)} \frac{\ua^{\crit}}{d_{g}(x,x_0)^s} \, dv_g=0. 
	\end{equation}
\end{step}
\noindent{\it Proof of Step \ref{step00}:} Using Step \ref{step0} and applying again Theorem \ref{theo1} with $z_{\alpha}=x_0$, we get the convergence of $\tilde{u}_{\alpha}$ (see \eqref{def:uta}). Now, we want to proof \eqref{m8}. We obtain by change of variable $X=\ma^{-1}\exp_{x_{0}}^{-1}(x)$ and the definition of $\uta$ in \eqref{def:uta} that,
\begin{equation*}\
\int_{B_{R\ma}(x_0)} \frac{\ua^{\crit}}{d_{g}(x,x_0)^s} \, dv_g=\int_{B_{R}(0)} \frac{\utza^{\crit}}{|X|^s} \, dv_{\tilde{g}_{\alpha}},
\end{equation*} 
where $\tilde{g}_\alpha$ is defined in \eqref{def:tga}. Therefore, applying Lebesgue's convergence Theorem and using the uniform convergence in $C^{0,\beta}_{loc}(\rr^n)$, for all $\beta \in \left( 0,\min{1,2-s}\right)$ of \eqref{convergence}, 
\begin{eqnarray}
\lim_{R\to +\infty}\lim_{\alpha\to +\infty}\int_{B_{R\ma}(x_0)} \frac{\ua^{\crit}}{d_{g}(x,x_0)^s} \, dv_g
&=&\lim_{R\to +\infty}\int_{B_{R}(0)} \frac{\tilde{u}^{\crit}}{|X|^s} \, dX\nonumber\\
&=&	\int_{\rr^n} \frac{\tilde{u}^{\crit}}{|X|^s} \, dX=1, \, \bb{ thanks to } \eqref{eq:norme1}. \label{m44}
\end{eqnarray}
From $\vv u_{\alpha}\vv_{\crit,s}^{\crit}=1$ and \eqref{m44}, we conclude that \eqref{m8}. This ends Step \ref{step00}.\qed

\begin{step}\label{cv:d12}
We claim that for any $R>0$, 
\begin{equation}\label{conv:D12F}
\utza\to \tilde{u}\hbox{ in }H_{1}^2(B_R(0))\hbox{ as }\alpha\to +\infty.
\end{equation}
\end{step}
\noindent{\it Proof of Step \ref{cv:d12}:} We rewrite \eqref{eq:utza} as 
 $$\Delta_{\tilde{g}_{\alpha}}\tilde{u}_{\alpha}=f_{\alpha}:=\lambda_{\alpha}\frac{ \tilde{u}_{\alpha}^{\crit-1}}{|X|^s}-\tilde{a}_{\alpha} \tilde{u}_{\alpha}.$$
Thanks to \eqref{convergence}, we get $f_{\alpha}(X) \to f(X)=\mu_{s,0}(\rr^n)\frac{\tilde{u}^{\crit-1}(X)}{|X|^s}$ in $C^{0,\beta}_{loc}(\rr^n \backslash \{0\})$, for all $\beta \in \left( 0,\min\{1,2-s\}\right)$. For any $R>0$, we have 
\begin{equation*}
	\vv f_\alpha\vv_{L^p(B_{2R}(0))} \leq \vv |X|^{-s}\vv_{L^{p}(B_{2R}(0))}\vv \utza\vv_{L^{\infty}(B_{2R}(0))}.
	\end{equation*}
It follows from \eqref{convergence} that $(\utza)_\alpha$ is bounded in $L^{\infty}_{loc}$. Since   $X\to |X|^{-s}\in L^{p}_{loc}(\rr^n)$ for $1<p<\frac{n}{s}$, then for such $p$, we have that $(f_\alpha)_\alpha$ is bounded in $L^{p}(B_{2R}(0))$. Using standard elliptic theory (see for instance \cite{GT}), we infer that
$$\vv \utza\vv_{H_2^p(B_0(R))}\leq C\left(\vv f_{\alpha}\vv_{L^{p}(B_{2R}(0))}+\vv \utza \vv_{L^p(B_{2R}(0))}\right).$$
Define now $p^\star$ such that $\frac{1}{p^\star}=\frac{1}{p}-\frac{1}{n}$. If $p^\star\leq 0$, $H_1^p(B_R(0))$ is compactly embedded in $L^2(B_R(0))$. Now, if $p^\star>0$, we have $H_1^p(B_R(0))$ is compactly embedded in $L^{q}(B_R(0))$ for $1\leq q<p^\star$ and $L^{2^\star}(B_R(0))\hookrightarrow L^2(B_R(0))$ iff $2\leq p^\star\iff p\geq \frac{2n}{n+2}$.
But, $s\in(0,2)$ then there exists $p>1$ such that $p\in (\frac{2n}{n+2}, \frac{n}{s})$ and then $(\utza)$ is bounded in $H_2^p(B_0(R))\hookrightarrow H_1^2(B_0(R))$. Since the embedding is compact, up to extraction, we get \eqref{conv:D12F} and ends Step \ref{cv:d12}.\qed

\begin{step}\label{step3} We claim that there exists $C>0$ such that 
\begin{equation*}
d_g(x,x_0)^{\frac{n-2}{2}} \ua(x)\leq C \bb{ for all } x\in M \bb{ and } \alpha>0.
\end{equation*}
\end{step}  
\noindent{\it Proof of Step \ref{step3}:} We follow the arguments of Jaber \cite{J2} (see also Druet \cite{D} and  Hebey \cite{H}). We argue by contradiction and assume that there exists $(y_{\alpha})_{\alpha}\in M$ such that 
 \begin{equation}\label{m10}
\sup_{x\in M} d_g(x,x_0)^{\frac{n-2}{2}}\ua(x)= d_g(y_{\alpha},x_0)^{\frac{n-2}{2}}\ua(y_{\alpha}) \to +\infty \bb{ as } \alpha \to +\infty.
 \end{equation}
Since $M$ is compact, we then get that $\lim_{\alpha\to +\infty} u_{\alpha}(y_{\alpha})=+\infty$. Thanks again to Lemma \ref{co2}, we obtain that, up to a subsequence,
 \begin{equation}\label{converge1}
 	\lim_{\alpha\to +\infty} y_{\alpha}=x_0.
 \end{equation}
 \noindent For $\alpha>0$, we define $\nu_{\alpha}:=\ua(y_{\alpha})^{-\frac{2}{n-2}}$, and then 
 \begin{equation}\label{converge2}
 	\nu_{\alpha}\to 0 \bb{ as } \alpha\to +\infty.
 \end{equation}
 We adopt the following notation: $(\theta_R)$ will denote any quantity such that 
 $$\lim_{R\to +\infty}\theta_{R}=0.$$
 
\noindent We claim that 
 \begin{eqnarray}\label{m9}
\int_{B_{\nu_{\alpha}}(y_{\alpha})}\frac{u_{\alpha}^{\crit }}{d_g(x,x_0)^s} \, dv_g=o(1)\hbox{ as }\alpha\to \infty.
 \end{eqnarray}
 \noindent{\it Proof of \eqref{m9}:} We fix $\delta>0$ and for any $R>0$,  
  \begin{eqnarray*}
  \int_{ B_{\delta}(x_0)\backslash B_{R\ma}(x_0)}\frac{u_{\alpha}^{\crit }}{d_g(x,x_0)^s} \, dv_g\leq \int_{ M\backslash B_{R\ma}(x_0)}\frac{u_{\alpha}^{\crit }}{d_g(x,x_0)^s} \, dv_g.
  \end{eqnarray*}
Therefore, it follows from the equation \eqref{m8} in the Step \ref{step00} that,
\begin{equation}\label{m33}
\int_{ B_{\delta}(x_0)\backslash B_{R\ma}(x_0)}\frac{u_{\alpha}^{\crit }}{d_g(x,x_0)^s} \, dv_g=\td+o(1).
\end{equation}
On the other hand, equations \eqref{converge1} and \eqref{converge2} yield
$B_{\nu_{\alpha}}(y_{\alpha})\backslash B_{\delta}(x_0)=\emptyset,$
and
\begin{eqnarray}
	\int_{B_{\nu_{\alpha}}(y_{\alpha})  }\frac{u_{\alpha}^{\crit }}{d_g(x,x_0)^s} \, dv_g&=& \int_{B_{\nu_{\alpha}}(y_{\alpha})\cap B_{\delta}(x_0) }\frac{u_{\alpha}^{\crit }}{d_g(x,x_0)^s} \, dv_g\nonumber\\
	&=&\int_{B_{\nu_{\alpha}}(y_{\alpha})\cap\left( B_{\delta}(x_0)\backslash B_{R\ma}(x_0)\right)  }\frac{u_{\alpha}^{\crit }}{d_g(x,x_0)^s} \, dv_g\nonumber\\
	&&+ \int_{B_{\nu_{\alpha}}(y_{\alpha})\cap B_{R\ma}(x_0) }\frac{u_{\alpha}^{\crit }}{d_g(x,x_0)^s} \, dv_g\nonumber\\
	&\leq& \td+ \int_{B_{\nu_{\alpha}}(y_{\alpha})\cap B_{R\ma}(x_0) }\frac{u_{\alpha}^{\crit }}{d_g(x,x_0)^s} \, dv_g+o(1) \bb{ (with \eqref{m33})}\label{m16}
\end{eqnarray}
We now distinguish two cases:

\medskip\noindent {\bf Case 1:} If $B_{\nu_{\alpha}}(y_{\alpha})\cap B_{R\ma}(x_0)=\varnothing$, then  \eqref{m9} is  a consequence of \eqref{m8}.

\medskip\noindent {\bf Case 2:} If $B_{\nu_{\alpha}}(y_{\alpha})\cap B_{R\ma}(x_0)\neq\varnothing$. Then,
\begin{equation}\label{m11}
d_g(y_{\alpha},x_0)\leq \nua+R\ma.
\end{equation}
 It follows from  the definition of $\nua$  and \eqref{m10} that,
 \begin{equation}\label{m12}
 \lim_{\alpha\to +\infty}\frac{\nua}{d_g( y_{\alpha},x_0)}=0.
 \end{equation}
Combining the equations \eqref{m11} and \eqref{m12},
\begin{equation}\label{m13}
d_g( y_{\alpha},x_0)=O(\ma) \bb{ and } \nua=o(\ma) \bb{ as } \alpha \to +\infty.
\end{equation}
We now consider an exponential chart $\left( \Omega_0, \exp_{x_0}^{-1}\right) $ centered at $x_0$ such that $\exp_{x_0}^{-1}(\Omega_0)=B_{r_0}(0)$, $r_0\in (0, i_g(M))$. We take $\tilde{Y}_{\alpha}=\ma^{-1}\exp_{x_0}^{-1}(y_{\alpha})$. By compactness arguments, there exists $c>1$ such that for all $X, Y\in \rr^n$, $\ma|X|, \ma|Y|<r_0$, 
$$ \frac{1}{c}|X-Y|\leq d_{\gta}(X,Y)\leq c |X-Y|.$$
Therefore, we have:
\begin{equation*}
\ma^{-1}\exp_{x_0}^{-1}\left(B_{\nua}(y_{\alpha}) \right)\subset B_{c\frac{\nua}{\ma}}(\tilde{Y}_{\alpha}).
\end{equation*}
And by equation \eqref{m13},
\begin{equation*}
	\tilde{Y}_{\alpha}= O\left(  d_{\gta}(\tilde{Y}_{\alpha},0)\right) =O\left( \ma^{-1}d_g(y_{\alpha},x_0) \right)=O(1).
\end{equation*}
It follows from \eqref{m12}, \eqref{m13} and the change of variables $X=\ma^{-1}\exp_{x_0}^{-1}(x)$ that,
\begin{eqnarray*}
\int_{B_{\nu_{\alpha}}(y_{\alpha})\cap B_{R\ma}(x_0) }\frac{u_{\alpha}^{\crit }}{d_g(x,x_0)^s} \, dv_g &\leq& \int_{\ma^{-1}\exp_{x_0}^{-1}\left( B_{\nu_{\alpha}}(y_{\alpha})\right) }\frac{\uta^{\crit }}{d_{\gta}(X,0)^s} \, dv_{\gta}\\
&\leq& \int_{B_{c\frac{\nua}{\ma}}(\tilde{Y}_{\alpha}) }\frac{\uta^{\crit }}{d_{\gta}(X,0)^s} \, dv_{\gta}\\
\end{eqnarray*}
It follows from the equation $\nua=o(\ma)$ and Lebesgue's convergence Theorem,
\begin{eqnarray*}
	\lim_{\alpha\to +\infty}\int_{B_{\nu_{\alpha}}(y_{\alpha})\cap B_{R\ma}(x_0) }\frac{u_{\alpha}^{\crit }}{d_g(x,x_0)^s} \, dv_g =0.
\end{eqnarray*}
Therefore, combining this with \eqref{m16}, we conclude \eqref{m9}. This proves the claim.\qed

\medskip\noindent We take now a family $\left(\Omega_{\alpha}, \exp_{y_{\alpha}}^{-1} \right)_{\alpha>0}$ of exponential charts centered at $y_{\alpha}$. Set $r_0\in \left( 0, i_g(M)\right)$,  we define
$$\hat{u}_{\alpha}(X)=\nua^{\frac{n-2}{2}}u_{\alpha}(\exp_{y_{\alpha}}(\nua X)) \bb{ on } B_{r_0\nua^{-1}}(0)\subset \rr^n,$$
and the metric, 
$$\hat{g}_{\alpha}(X)=\exp_{y_{\alpha}}^{\star}g(\nua X) \bb{ on } \rr^n.$$
Since $u_{\alpha}$ verifies the equation \eqref{m2}, we get $\hat{u}_{\alpha}$ verifies also weakly  
$$\Delta_{\hat{g}_{\alpha}}\hat{u}_{\alpha}+\hat{a}_{\alpha} \hat{u}_{\alpha}=\lambda_{\alpha}\frac{ \hat{u}_{\alpha}^{\crit-1}}{d_{\hat{g}_{\alpha}}(X,X_{0,\alpha})^s} \bb{ in }\rr^n,$$
where $\hat{a}_{\alpha}(X):=\nua^2 a_{\alpha}(\exp_{y_{\alpha}}(\nua X))\to 0 $ as $\alpha\to +\infty$ and $ X_{0,\alpha}=\na ^{-1} \exp_{y_{\alpha}}^{-1} (x_0)$.	

\medskip\noindent We claim that 
\begin{equation}\label{lim:nonzero}
 \hat{u}_{\alpha}\to \hat{u}\not\equiv 0 \bb{ in } C^0_{loc}(\rr^n) \bb{ as } \alpha \to +\infty.
\end{equation} 
We prove \eqref{lim:nonzero}. Using the definition of $\hat{u}_{\alpha}$ and the equation \eqref{m10}, we get 
	\begin{eqnarray}\label{m17}
		 \hat{u}_{\alpha}(X)\leq \left( \frac{d_g(x_0,y_{\alpha})}{d_g(\exp_{y_{\alpha}}(\nua X),x_0)}\right)^{\frac{n-2}{2}} \bb{ for all } X\in B_{r_0\nua^{-1}}(0).
		\end{eqnarray}
	On the other hand, from the triangular inequality and for any $X\in B_{R}(0)$, we obtain that
\begin{align*}
	d_g(\exp_{y_{\alpha}}(\nua X),x_0)&\geq d_g(x_0,y_{\alpha}) -d_g(\exp_{y_{\alpha}}(\nua X),y_{\alpha})\\
&=d_g(x_0,y_{\alpha}) - \nua |X|\\
&\geq d_g(x_0,y_{\alpha}) -\nua R.
\end{align*}
Therefore, with the equation \eqref{m17}, we have for all $X\in B_{R}(0)$ that,
\begin{eqnarray*}
	\hat{u}_{\alpha}(X)&\leq& \left( \frac{1}{1-\frac{\nua R}{d_g(x_0,y_{\alpha})}}\right)^{\frac{n-2}{2}}.
\end{eqnarray*}
Moreover, with \eqref{m11}, we obtain for all $X\in B_{R}(0)$, that $\hat{u}_{\alpha}(X)\leq 1+o(1)$ in $C^{0}(B_{R}(0))$. 
\noindent Using again the definition $\nua$, we have $\hat{u}_{\alpha}(0)=1$ for all $\alpha>0$. Elliptic Theory yields $\hat{u}_{\alpha}\to \hat{u} \bb{ in } C_{loc}^0(\rr^n)$ and we have also that $\hat{u}(0)=\lim_{\alpha\to +\infty}\hat{u}_{\alpha}(0)=1$. This yields \eqref{lim:nonzero} and the claim is proved.

\medskip\noindent Take $X=\nua^{-1}\exp^{-1}_{y_{\alpha}}(x)$ and from the definition of $\hat{u}_{\alpha}$, we infer that 
\begin{equation*}
\int_{B_{1}(0)}\frac{\hat{u}_{\alpha}^{\crit }}{d_{\hat{g}_{\alpha}}(X,X_{0,\alpha})^s} \, dv_{\hat{g}_{\alpha}}=\int_{B_{\nua}(y_{\alpha})}\frac{u_{\alpha}^{\crit }}{d_g(x,x_0)^s} \, dv_g.
\end{equation*}
Therefore, using Lebesgue's convergence Theorem and \eqref{m9}, we obtain that 
\begin{align*}
\int_{B_{1}(0)}\frac{\hat{u}^{\crit }}{|X|^s} \, dX
&=\lim_{\alpha\to +\infty}\int_{B_{\nua}(y_{\alpha})}\frac{u_{\alpha}^{\crit }}{d_g(x,x_0)^s} \, dv_g\\
&=0.
\end{align*}
with $\theta_R\to 0$ as $R \to +\infty$. Which yields $\hat{u}\equiv 0$ in $B_1(0)$, contradicting $\hat{u} \in C^0\left( B_1(0)\right) $ and $\hat{u}(0)=1$.  This completes the proof of Step \ref{step3}.\qed

\begin{step}\label{step4}
We claim  that
\begin{equation*}
\lim_{R\to +\infty}	\lim_{\alpha\to +\infty} \sup_{x\in M\backslash B_{R\ma}(x_0)}d_g(x,x_0)^{\frac{n-2}{2}}u_{\alpha}(x)=0.
\end{equation*}
\end{step} 
\noindent{\it Proof of Step \ref{step4}:} The proof is a refinement of Step \ref{step3}. We omit it and we refer to \cite{DD} and Chapter 4 in Druet-Hebey-Robert \cite{DHR} where the case $s=0$ is dealt with.\qed

	\section{Refined blowup analysis: proof of Theorem \ref{prop1}}\label{sec:estimate}
We let $(u_\alpha)_\alpha\in H_1^2(M)$, $(a_\alpha)_\alpha\in C^1(M)$, $a_\infty\in C^1(M)$, $(\ll_\alpha)_\alpha\in\rr$ be such that \eqref{m0}-\eqref{m00} hold. The next Step towards the proof of Theorem \ref{prop1} is the following:

	\begin{step}\label{step5} We claim that there exists $\epsilon_0>0$ such that for any $\epsilon \in (0,\epsilon_0)$,  there exists $C_{\epsilon}>0$  such that 
	\begin{eqnarray}\label{m40}
	u_{\alpha}(x) \leq C_{\epsilon}  \frac{\ma^{\frac{n-2}{2}-\epsilon}}{d_g(x,x_0)^{n-2-\epsilon}} \bb{ for all } x \in M\backslash B_{R \ma}(x_0).
	\end{eqnarray}
	\end{step}
	
\noindent{\it Proof of Step \ref{step5}:}	Let $G$ be the Green function on $M$ of $\Delta_g +(a_{\infty}-\xi)$ where $\xi >0$. Up to taking $\xi$ small enough, the operator is coercive and the $G_{x_0}:=G(x_0,\cdot)$ is defined on $M\backslash \{x_{0}\}$. In others words, $G_{x_0}$ satisfies
	\begin{equation}\label{m35}
	\Delta_g G_{x_0}+(a_{\infty}-\xi)G_{x_0}=0 \bb{ in } M\backslash \{x_0\}.
	\end{equation}
	Estimates of the Green's function (see Robert\cite{RG}) yield for $\delta_0>0$ small the existence of $C_i>0$ for $i=1,2,3$ such that
	\begin{equation}\label{m19}
	C_2\, d_g(x,x_{0})^{2-n}\leq G_{x_0}(x)\leq C_1\, d_g(x,x_{0})^{2-n},
	\end{equation}
and,
\begin{eqnarray}\label{m37}
|\nabla G_{x_0}(x)|_g\geq C_3\, d_g(x,x_0)^{1-n},
\end{eqnarray}
for all $\alpha\in \nn$ and all $x\in B_{\delta_0}(x_0)\backslash \{x_0\}$.
	Define the operator 
	$$ M_{g,\alpha}:= \Delta_g+ a_{\alpha}-\ll_{\alpha}\frac{u_{\alpha}^{\crit-2}}{d_g(x,x_0)^s}.$$

\medskip\noindent{\it Step \ref{step5}.1:} We claim that there exists $\nu_0\in (0,1)$ and $R_0>0$ such that for any $\nu\in (0,\nu_0)$ and $R>R_0$, we have that 
	\begin{equation}\label{m38}
		M_{g,\alpha} G_{x_0}^{1-\nu}>0 \bb{ for all } x\in M\backslash B_{R\ma}(x_0).
	\end{equation}
\noindent{\it Proof of Step \ref{step5}.1:} With \eqref{m35}, we get that 
		\begin{align*}
		\frac{M_{g,\alpha} G_{x_0}^{1-\nu}}{G_{x_0}^{1-\nu}}(x)&=a_{\alpha}-a_{\infty}+\nu\left( a_{\infty}-\xi\right) +\xi +\nu\left( 1-\nu\right)\left| \frac{\nabla G_{x_0} }{G_{x_0}} \right|_g^2-\ll_{\alpha}\frac{u_{\alpha}^{\crit-2}}{d_g(x,x_0)^s},
		\end{align*}
		for all $x\in M\backslash \{x_0\}$. Using again \eqref{m0}, there exists $\alpha_0$  for all $\alpha>\alpha_0$ such that 
		\begin{equation*}
		a_{\alpha}(x)-a_{\infty}(x)\geq -\frac{\xi}{2} \bb{ for all } x\in M.
		\end{equation*}
		Take now $\nu_0\in (0,1)$ and we let $\nu\in (0,\nu_0)$, we get that 
		\begin{align}\label{m36}
		\frac{M_{g,\alpha} G_{x_0}^{1-\nu}}{G_{x_0}^{1-\nu}}(x)&\geq \frac{\xi}{4} +\nu\left( 1-\nu\right)\left| \frac{\nabla G_{x_0}}{G_{x_0}} \right|_g^2-\ll_{\alpha}\frac{u_{\alpha}^{\crit-2}}{d_g(x,x_0)^s}.
		\end{align}
		Fix $\rho>0$, it follows from the result of the Step \ref{step4} that there exists $R_0>0$ such that for any $R>R_0$ and for $\alpha> 0$ large enough, we obtain that 
		\begin{equation}\label{m21}
		d_g(x,x_0)^{\frac{n-2}{2}} u_{\alpha}(x)\leq \rho  \bb{ for } x\in M\backslash B_{R \ma}(x_0).
		\end{equation}
		We let $\nu\in (0,\nu_0)$ and $R>R_0$. We first let $x\in M$ such that $d_g(x,x_0)\geq \delta_0$, then from Corollary \ref{co2}  
		\begin{equation}\label{eq:conv0}
			\lim_{\alpha \to +\infty} u_{\alpha}(x)=0 \bb{ in } M\backslash B_{\delta_0}(x_0).
		\end{equation}
		From \eqref{m36} and \eqref{m1}, we have that 
		\begin{align*}
		\frac{M_{g,\alpha} G_{x_0}^{1-\nu}}{G_{x_0}^{1-\nu}}(x)&\geq \frac{\xi}{4} -2\mu_s(\rr^n)\frac{u_{\alpha}(x)^{\crit-2}}{\delta_0^s},
		\end{align*}
		and $\alpha\in \nn$. The convergence in \eqref{eq:conv0} yields \eqref{m38} when $d_g(x,x_0)\geq \delta_0$ for $\alpha$ large enough. 

\smallskip\noindent We now take $x\in B_{\delta_0}(x_0)\backslash B_{R\ma}(x_0)$. It follows from \eqref{m36}, \eqref{m21}, \eqref{m1}, \eqref{m19} and \eqref{m37} that,  
		\begin{align*}\label{m39}
		\frac{M_{g,\alpha} G_{x_0}^{1-\nu}}{G_{x_0}^{1-\nu}}(x)&\geq \frac{1}{d_g(x,x_0)^2}\left(  \nu\left( 1-\nu\right)\left( \frac{C_3}{C_1} \right)^2-2\mu_s(\rr^n)\rho^{\crit-2}\right) .
		\end{align*}
		Up to taking $\rho>0$ small enough, we then obtain \eqref{m38} when $x\in B_{\delta_0}(x_0)\backslash B_{R\ma}(x_0)$. This ends Step \ref{step5}.1.\qed 
	
\medskip\noindent{\it Step \ref{step5}.2:} We claim that there exists $C_R>0$ such that 
 \begin{eqnarray*}
u_{\alpha}(x) \leq C_R\,\mu_{\alpha}^{\frac{n-2}{2}-\nu(n-2)}G_{x_0}(x)^{1-\nu} \bb{ for any } x \in \partial B_{R\ma}(x_0) \bb{ and } \alpha \in \nn.
\end{eqnarray*}

\medskip\noindent{\it Proof of Step \ref{step5}.2:} It follows from \eqref{def:uta}, \eqref{convergence} and  \eqref{m19} that,
\begin{eqnarray*}
		u_{\alpha}(x)&\leq &C\,  \mu_{\alpha}^{-\frac{n-2}{2}} \\
		&=& C\,  \mu_{\alpha}^{-\frac{n-2}{2}} d_g(x,x_{0})^{-(2-n)(1-\nu)} d_g(x,x_{0})^{(2-n)(1-\nu)}\\
		&\leq & C\,C_2^{\nu-1}\, \mu_{\alpha}^{-\frac{n-2}{2}}d_g(x,x_{0})^{(n-2)(1-\nu)} G_{x_0}(x)^{1-\nu} \\
		&\leq&  C\,C_2^{\nu-1}\, R^{(n-2)(1-\nu)}\mu_{\alpha}^{\frac{n-2}{2}-\nu(n-2)}G_{x_0}(x)^{1-\nu}.
\end{eqnarray*}
This ends Step \ref{step5}.2.\qed

\medskip\noindent{\it Step \ref{step5}.3:} We claim that
\begin{eqnarray*}
 	u_{\alpha}(x) \leq C_R\,\mu_{\alpha}^{\frac{n-2}{2}-\nu(n-2)}G_{x_0}(x)^{1-\nu} \bb{ for any } x \in M\backslash B_{R\ma}(x_0).
 	\end{eqnarray*}
\medskip\noindent{\it Proof of Step \ref{step5}.3:} We define $ v_{\alpha}:= C_R\,\mu_{\alpha}^{\frac{n-2}{2}-\nu(n-2)}G_{x_0}(x)^{1-\nu}-u_{\alpha}$. Since $u_{\alpha}$ verifies \eqref{m2} and by  \eqref{m38}, we observe that 

\begin{eqnarray*}
 		M_{g,\alpha} v_{\alpha}&=&C_R\,\mu_{\alpha}^{\frac{n-2}{2}-\nu(n-2)}	M_{g,\alpha}G_{x_0}^{1-\nu} -	M_{g,\alpha}u_{\alpha}\\
 		&=&C_R\,\mu_{\alpha}^{\frac{n-2}{2}-\nu(n-2)}	M_{g,\alpha}G_{x_0}^{1-\nu}>0  \bb{ in }  M\backslash B_{R \ma}(x_0).
 \end{eqnarray*}
Then Step \ref{step5}.3 follows from this inequality, Step \ref{step5}.2 and the comparison principle (See Berestycki--Nirenberg-Varadhan \cite{BNV}). This ends Step \ref{step5}.3.\qed

\medskip\noindent We are in position to finish the proof of Step \ref{step5}. Step \ref{step5}.3  and \eqref{m19} yield
	\begin{equation}\label{m43}
	u_{\alpha}(x) \leq C_{R}^{\prime}  \frac{\ma^{\frac{n-2}{2}-\nu(n-2)}}{d_g(x,x_0)^{(n-2)(1-\nu)}} \bb{ for all } x \in M\backslash B_{R\ma}(x_0).
\end{equation}
\noindent On the other hand, in \eqref{m00}, for $x\in B_{R\ma}(x_0)\setminus\{x_0\}$ and $\nu \in (0,\nu_0)$
\begin{eqnarray*}
	u_{\alpha}(x)&\leq& \ma ^{-\frac{n-2}{2}}\leq \ma ^{\frac{n-2}{2}-\nu(n-2)} \ma ^{(\nu-1)(n-2)}\nonumber\\
&\leq& R^{(1-\nu)(n-2)} \frac{ \ma ^{\frac{n-2}{2}-\nu(n-2)}}{d_g(x,x_0)^{(1-\nu)(n-2)}} \bb{ for all } x \in B_{R\ma}(x_0). 
\end{eqnarray*}
Up to taking $C^{\prime}_R$ larger and $\epsilon=(n-2)\nu$, by \eqref{m43}, we get  inequality \eqref{m40}. This ends Step \ref{step5}.\qed

\begin{step}\label{step6}
 We claim that there exists $C>0$ such that 
\begin{equation}\label{m411}
d_g(x,x_0)^{n-2}u_{\alpha}(x_{\alpha})u_{\alpha}(x) \leq C  \bb{ for all } x \in M.
\end{equation}
\end{step}
\noindent{\it Proof of Step \ref{step6}:} We let $(y_{\alpha})_\alpha\in M$ be such that $$ \sup_{x\in M}d_g(x,x_0)^{n-2}u_{\alpha}(x_{\alpha})u_{\alpha}(x)=d_g(y_{\alpha},x_0)^{n-2}u_{\alpha}(x_{\alpha})u_{\alpha}(y_{\alpha}).$$ The claim is equivalent to proving that for any $y_{\alpha}$, we have that
$$ d_g(y_{\alpha},x_0)^{n-2}u_{\alpha}(x_{\alpha})u_{\alpha}(y_{\alpha})=O(1) \bb{ as } \alpha \to +\infty.$$
We distinguish two cases:

\medskip\noindent {\bf Case 1:}  We assume that $ d_g(y_{\alpha},x_0)=O(\ma)$ as $\alpha \to +\infty$. Therefore, it follows from the definition of $\ma$ that
\begin{align*}
d_g(y_{\alpha},x_0)^{n-2}u_{\alpha}(x_{\alpha})u_{\alpha}(y_{\alpha})&\leq C \ma ^{n-2}u_{\alpha}^2(x_{\alpha})\leq C.
\end{align*}
This yields \eqref{m411}.

 \medskip\noindent {\bf Case 2:} We assume that 
\begin{equation}\label{m25}
 \lim_{\alpha\to +\infty} \frac{d_g(y_{\alpha},x_0)}{\ma}=+\infty.
\end{equation} 
Let  $G_{\alpha}$ be the Green's function of $\Delta_g+a_{\alpha}$ in $M$. Green's representation formula and standard estimates on the Green's function (see \eqref{m19} and Robert \cite{RG}) yield the existence of $C>0$ such that
\begin{align}
u_{\alpha}(y_{\alpha})&=\int_M G_{\alpha}(y_{\alpha},x) \ll_{\alpha}\frac{\ua^{\crit -1 }(x)}{d_g(x,x_0)^s} \, dv_g\nonumber\\
&\leq C \int_Md_g(x,y_{\alpha})^{2-n} \ll_{\alpha}\frac{\ua^{\crit -1 }(x)}{d_g(x,x_0)^s} \, dv_g.
\label{m301}
\end{align}
We fix $R>0$ and we write $M:= \cup_{i=1}^{4}\Omega_{i,\alpha}$ where 
\begin{eqnarray*}
	\Omega_{1,\alpha}:=B_{R\ma}(x_0) &\bb{ and }&	\Omega_{2,\alpha}:=\left\lbrace R \ma < d_g(x,x_0)< \frac{d_g(y_{\alpha},x_0)}{2}\right\rbrace, \\
	\Omega_{3,\alpha}:=\left\lbrace \frac{d_g(y_{\alpha},x_0)}{2} < d_g(x,x_0)<2 d_g(y_{\alpha},x_0)\right\rbrace &\bb{ and }& \Omega_{4,\alpha}:=\left\lbrace d_g(x,x_0)\geq  2d_g(y_{\alpha},x_0)\right\rbrace\cap M. \\
\end{eqnarray*}
\medskip\noindent{\it Step \ref{step6}.1:} We first deal with $\Omega_{1,\alpha}$.

\smallskip\noindent Using \eqref{m25}, we fix $C_0>R$. For  $\alpha$ large, we have that
	\begin{align*}
		d_g(y_{\alpha},x_0)&\geq C_0\, \ma \geq \frac{C_0}{R} \,  d_g(x,x_0)\hbox{ for all }x\in \Omega_{1,\alpha}.
	\end{align*}
Then since $C_0>R>1$, we get $d_g(x,y_{\alpha})\geq  \left( 1-\frac{R}{C_0}\right) 			d_g(y_{\alpha},x_0)$. 
	Therefore, we take $x=\exp_{x_{0}}(\ma X)$, then for $R>1$ there exists $C>0$ such that  
\begin{align}
\left|  \int_{\Omega_{1,\alpha}}d_g(x,y_{\alpha})^{2-n} \frac{\ua^{\crit -1 }(x)}{d_g(x,x_0)^s} \, dv_g\right| &\leq C\, 	d_g(y_{\alpha},x_0)^{2-n}\int_{\Omega_{1,\alpha}}\frac{\ua^{\crit -1 }(x)}{d_g(x,x_0)^s}\, dv_g\nonumber\\
&\leq C\,	\ma ^{\frac{n-2}{2}} 	d_g(y_{\alpha},x_0)^{2-n}\int_{B_R(0)}\frac{\utza^{\crit -1 }(X)}{|X|^s}\, dv_{\tilde{g}_{\alpha}},\label{m451}
\end{align}
where $\utza$, $\tilde{g}_{\alpha}$ are defined in \eqref{def:uta}, \eqref{def:tga}. Since $\utza\leq 1$, by applying Lebesgue's Convergence Theorem and thanks to Step \ref{step00}, we get that 
\begin{eqnarray}\label{m471}
\lim_{\alpha\to +\infty}\int_{B_R(0)}\frac{\utza^{\crit -1 }(X)}{|X|^s}\, dv_{\tilde{g}_{\alpha}}=\int_{B_R(0)}\frac{\ut^{\crit -1 }}{|X|^s}\, dX. 
\end{eqnarray}
Combining \eqref{m451} and \eqref{m471} yields
\begin{eqnarray}\label{m261}
\left|  \int_{\Omega_{1,\alpha}}d_g(x,y_{\alpha})^{2-n} \frac{\ua^{\crit -1 }(x)}{d_g(x,x_0)^s} \, dv_g\right|	&\leq C \, \ma ^{\frac{n-2}{2}}		d_g(y_{\alpha},x_0)^{2-n}.
\end{eqnarray}

\medskip\noindent{\it Step \ref{step6}.2:} We deal with $\Omega_{2,\alpha}$.

\smallskip\noindent Noting that $d_g(x,y_{\alpha})\geq d_g(y_{\alpha},x_0)- d_g(x,x_0)\geq \frac{1}{2}d_g(y_{\alpha},x_0)$ for all $x\in \Omega_{2,\alpha}$, we argue as in Step \ref{step6}.1  by using \eqref{m40} with $\epsilon>0$ small to get
\begin{align*}
&\left|  \int_{\Omega_{2,\alpha}}d_g(x,y_{\alpha})^{2-n} \frac{\ua^{\crit -1 }(x)}{d_g(x,x_0)^s} \, dv_g\right|\\
& \leq C \, 	d_g(y_{\alpha},x_0)^{2-n} \int_{\Omega_{2,\alpha}}\frac{\ua^{\crit -1 }(x)}{d_g(x,x_0)^s}\, dv_g\nonumber \\
& \leq C\,\ma^{(\frac{n-2}{2}-\epsilon)(\crit -1)}	d_g(y_{\alpha},x_0)^{2-n} \int_{\Omega_{2,\alpha}}d_g(x,x_0)^{-s-(n-2-\epsilon)(\crit-1)}\, dv_g\nonumber \\
& \leq C\,\ma^{(\frac{n-2}{2}-\epsilon)(\crit -1)}	d_g(y_{\alpha},x_0)^{2-n} \int_{M\backslash B_{R\ma }(x_0)}d_g(x,x_0)^{-s-(n-2-\epsilon)(\crit-1)}\, dv_g\nonumber \\
\end{align*}
\noindent Taking the change of variable $X=\exp_{x_{0}}^{-1}(x)$ and $\hat{g}=\exp_{x_{0}}^\star g$ on $\rr^n$, we get 
\begin{align*}
&\left|  \int_{\Omega_{2,\alpha}}d_g(x,y_{\alpha})^{2-n} \frac{\ua^{\crit -1 }(x)}{d_g(x,x_0)^s} \, dv_g\right|\nonumber\\
& \leq C\,\ma^{(\frac{n-2}{2}-\epsilon)(\crit -1)}d_g(y_{\alpha},x_0)^{2-n} \int_{\rr^n\backslash B_{R\ma }(0)}|X|^{-s-(n-2-\epsilon)(\crit-1)}\, dv_{\hat{g}}\\
& \leq C\,\ma^{(\frac{n-2}{2}-\epsilon)(\crit -1)}	d_g(y_{\alpha},x_0)^{2-n} \int_{\rr^n\backslash B_{R\ma }(0)}|X|^{-s-(n-2-\epsilon)(\crit-1)}\, dX\\
&\leq C\,\ma^{\frac{n-2}{2}}	d_g(y_{\alpha},x_0)^{2-n} \int_{R}^{+\infty}r^{s-2+\epsilon(\crit-1)-1}\, dr.
\end{align*}
Hence for $\epsilon>0$ sufficiently small, we have that 
\begin{equation}\label{m271}
\left|  \int_{\Omega_{2,\alpha}}d_g(x,y_{\alpha})^{2-n} \frac{\ua^{\crit -1 }(x)}{d_g(x,x_0)^s} \, dv_g\right|	\leq C_R\,\ma^{\frac{n-2}{2}}	d_g(y_{\alpha},x_0)^{2-n}  ,
\end{equation}
as $\alpha\to +\infty$, where $C_R\to 0$ as $R\to +\infty$.

\medskip\noindent{\it Step \ref{step6}.3:} We deal with $\Omega_{3,\alpha}$. For $\epsilon>0$  small in the control \eqref{m40}, we get 
\begin{align}
&\left|  \int_{\Omega_{3,\alpha}}d_g(x,y_{\alpha})^{2-n} \frac{\ua^{\crit -1 }(x)}{d_g(x,x_0)^s} \, dv_g\right|\nonumber\\
&\leq C\, \ma^{(\frac{n-2}{2}-\epsilon)(\crit -1)}	d_g(y_{\alpha},x_0)^{-s-(n-2-\epsilon)(\crit-1)} \int_{\Omega_{3,\alpha}}d_g(x,y_{\alpha})^{2-n}\, dv_g\nonumber.
\end{align}
It follows from the change of variable $x=\exp_{x_{0}}(X)$ and $y_{\alpha}=\exp_{x_{0}}(Y_{\alpha})$ that,
\begin{align*}
	&\left|  \int_{\Omega_{3,\alpha}}d_g(x,y_{\alpha})^{2-n} \frac{\ua^{\crit -1 }(x)}{d_g(x,x_0)^s} \, dv_g\right|\nonumber\\
	&\leq C\, \ma^{(\frac{n-2}{2}-\epsilon)(\crit -1)}d_g(y_{\alpha},x_0)^{-s-(n-2-\epsilon)(\crit-1)} \int_{\frac{1}{2}|Y_{\alpha}|<|X|<2|Y_{\alpha}|}|X-Y_{\alpha}|^{2-n}\, dv_{\hat{g}}\nonumber\\
	&\leq C\, \ma^{(\frac{n-2}{2}-\epsilon)(\crit -1)}d_g(y_{\alpha},x_0)^{-s-(n-2-\epsilon)(\crit-1)} \int_{\frac{1}{2}|Y_{\alpha}|<|X|<2|Y_{\alpha}|}|X-Y_{\alpha}|^{2-n}\, dX\nonumber\\
	&\leq C\, \ma^{(\frac{n-2}{2}-\epsilon)(\crit -1)}d_g(y_{\alpha},x_0)^{-s-(n-2-\epsilon)(\crit-1)} |Y_{\alpha}|^{2}\int_{\frac{1}{2}|<|X|<2}\left|X-\frac{Y_\alpha}{|Y_\alpha|}\right|^{2-n}\, dX\nonumber\\
		&\leq C\, \ma^{(\frac{n-2}{2}-\epsilon)(\crit -1)}d_g(y_{\alpha},x_0)^{-s-(n-2-\epsilon)(\crit-1)} d_g(y_{\alpha},x_0)^{2}\int_{|X|<2}\left|X-\frac{Y_\alpha}{|Y_\alpha|}\right|^{2-n}\, dX\nonumber\\
			&\leq C\, \ma^{(\frac{n-2}{2}-\epsilon)(\crit -1)}d_g(y_{\alpha},x_0)^{2-n-\frac{n-2}{2}(\crit-2)+\epsilon(\crit-1)} \int_{|X|<3}|X|^{2-n}\, dX\nonumber\\
	&\leq C\,\ma^{\frac{n-2}{2}}d_g(y_{\alpha},x_0)^{2-n}\left( \frac{\ma}{d_g(y_{\alpha},x_0)}\right) ^{(\frac{n-2}{2})(\crit-2)- \epsilon(\crit-1)}.
\end{align*}
Just take $\epsilon>0$ small, hence $(\frac{n-2}{2})(\crit-2)- \epsilon(\crit-1)>0$  and we obtain that, 
\begin{align}\label{m281}
\left|  \int_{\Omega_{3,\alpha}}d_g(x,y_{\alpha})^{2-n} \frac{\ua^{\crit -1 }(x)}{d_g(x,x_0)^s} \, dv_g\right|
\leq C\,\ma^{\frac{n-2}{2}}d_g(y_{\alpha},x_0)^{2-n}\left( \frac{\ma}{d_g(y_{\alpha},x_0)}\right) ^{(\frac{n-2}{2})(\crit-2)- \epsilon(\crit-1)}.
\end{align}

\medskip\noindent{\it Step \ref{step6}.4:} We deal with $\Omega_{4,\alpha}$. For $x\in\Omega_{4,\alpha}$, we have that
\begin{eqnarray*}
	d_g(x,y_{\alpha})\geq 	d_g(x,x_{0})- d_g(y_{\alpha},x_0)
	\geq \frac{1}{2} d_g(x,x_0).
\end{eqnarray*}
Taking $X=\exp_{x_{0}}^{-1}(x)$ and $Y_{\alpha}=\exp_{x_{0}}^{-1}(y_{\alpha})$, we obtain that
\begin{align*}
&\left|  \int_{\Omega_{4,\alpha}}d_g(x,y_{\alpha})^{2-n} \frac{\ua^{\crit -1 }(x)}{d_g(x,x_0)^s} \, dv_g\right| \\
&\leq C\,\ma^{(\frac{n-2}{2}-\epsilon)(\crit -1)} \int_{\Omega_{4,\alpha}}d_g(x,x_0)^{2-n-s-(n-2-\epsilon)(\crit-1)}\, dv_g \nonumber\\
&\leq C\,\ma^{(\frac{n-2}{2}-\epsilon)(\crit -1)}\int_{B_{\delta}(0)\backslash B_{2|Y_{\alpha}|}(0)}|X|^{2-n-s-(n-2-\epsilon)(\crit-1)}\, dv_{\hat{g}}\nonumber \\
&\leq C\,\ma^{(\frac{n-2}{2}-\epsilon)(\crit -1)}\int_{2|Y_{\alpha}|}^{+\infty}r^{-n+s+\epsilon(\crit-1)-1}\, dr.\nonumber 
\end{align*}
For $\epsilon>0$ sufficiently small, we get that 
\begin{equation} \label{m291}
\left|  \int_{\Omega_{4,\alpha}}d_g(x,y_{\alpha})^{2-n} \frac{\ua^{\crit -1 }(x)}{d_g(x,x_0)^s} \, dv_g\right| \leq C\, \ma^{\frac{n-2}{2}}d_g(y_{\alpha},x_0)^{2-n}\left( \frac{\ma}{d_g(x_0,y_{\alpha})}\right) ^{(\frac{n-2}{2})(\crit-2)- \epsilon(\crit-1)}.
\end{equation}
Plugging the equations \eqref{m261}-\eqref{m291} in \eqref{m301}, we get \eqref{m411}. This ends Step \ref{step6}.\qed

\begin{step}\label{stepf}
	We claim that there exists $C>0$, such that 
\begin{equation}\label{est:co}
	u_{\alpha}(x)\leq C \, \frac{\ma^{\frac{n-2}{2}}}{\ma^{n-2}+d_g(x,x_0)^{n-2}} \bb{ for all } x\in M.
\end{equation}
\end{step}
\noindent{\it Proof of Step \ref{stepf}:} Using \eqref{m411} and the definition of $\ma$ (see \eqref{m00}), we have

\begin{align*}
\left(\ma^{n-2}+d_g(x,x_0)^{n-2} \right)\ma^{-\frac{n-2}{2}} u_{\alpha}(x)&\leq  \ma^{\frac{n-2}{2}}u_{\alpha}(x)+C\\
&\leq \ma^{\frac{n-2}{2}}u_{\alpha}(x_{\alpha})+C\leq 1+C.
\end{align*}
This proves Theorem \ref{prop1} and ends Step \ref{stepf}. \qed

\medskip\noindent As a first remark, it follows from the definition \eqref{def:uta} and the pointwise control \eqref{ineq:up} of Theorem \ref{prop1} that
\begin{equation}\label{control:uta}
\uta(X)\leq \frac{C}{\left(1+|X|^2\right)^\frac{n-2}{2}}\hbox{ in }B_{\ma^{-1}\delta_0}(0).
\end{equation}
\begin{prop}\label{prop:estnabla}
	For all $R>0$, we claim that there exists $C>0$ such that 
	\begin{eqnarray}\label{m41}
	|\nabla u_{\alpha}(x)|_g \leq C \frac{\ma^{\frac{n-2}{2}}}{\left( d_g(x,x_0)^{2}+\ma^2\right)^{\frac{n-1}{2}} } \bb{ for all } x \in M\backslash B_{R\ma}(x_0),
	\end{eqnarray}	
	as $\alpha\to +\infty$.
\end{prop}
\noindent{\it Proof of Proposition  \ref{prop:estnabla}:} Let $(y_{\alpha})_\alpha\in M$  be such that $$ \sup_{x\in M}\left( d_g(x,x_0)^{n-1}+\ma^{n-1}\right) u_{\alpha}(x_{\alpha})|\nabla u_{\alpha}(x)|_g=\left( d_g(y_{\alpha},x_0)^{n-1}+\ma^{n-1}\right) u_{\alpha}(x_{\alpha})|\nabla u_{\alpha}(y_{\alpha})|.$$ The claim is equivalent to proving that for any $y_{\alpha}$, we have that
$$\left( d_g(y_{\alpha},x_0)^{n-1}+\ma^{n-1}\right)u_{\alpha}(x_{\alpha})|\nabla u_{\alpha}(y_{\alpha})|_g=O(1) \bb{ as } \alpha \to +\infty.$$
We let $G_{\alpha}$ be the Green's function of $\Delta_g+a_{\alpha}$ in $M$. Green's representation formula and the estimates \eqref{m37} yield $C>0$ such that
\begin{align*}
|\nabla u_{\alpha}(y_{\alpha})|&\leq \int_M |\nabla G_{\alpha}(y_{\alpha},x)|_g \ll_{\alpha}\frac{\ua^{\crit -1 }(x)}{d_g(x,x_0)^s} \, dv_g(x)\nonumber\\
&\leq C \int_Md_g(x,y_{\alpha})^{1-n} \frac{\ua^{\crit -1 }(x)}{d_g(x,x_0)^s} \, dv_g(x).
\end{align*}
More generally, we prove that for any sequence $(y_{\alpha})_\alpha\in M$ such that $d_g(y_{\alpha},x_0)\geq R\ma$ for some $R>0$, then there exist $C>0$  such that 
\begin{eqnarray*}
|\nabla u_{\alpha}(y_{\alpha})| \leq C\, \frac{\ma^{\frac{n-2}{2}}}{\ma^{n-1}+d_g(y_{\alpha},x_0)^{n-1}} \bb{ as } \alpha\to +\infty.
\end{eqnarray*}
Then using the pointwise estimates \eqref{ineq:up} the proof goes exactly as in Step \ref{step6}. \qed
\begin{prop}\label{lem2}
 We  claim that 
\begin{equation}\label{lim:u:G}
\lim_{\alpha\to +\infty}\frac{u_{\alpha}}{\ma^{\frac{n-2}{2}}}=d_n \,G_{x_0} \bb{ in } C^2_{loc}(M\backslash \{x_0\}),
\end{equation}
where,
\begin{equation}\label{dn}
d_n:=\mu_s(\rr^n)\int_{\rr^n} \frac{\ut^{\crit -1}}{|X|^s}\, dX,
\end{equation}
and $G_{x_0}$ is the Green's function for $\Delta_g+a_{\infty}$ on $M$ at $x_0$.
\end{prop}
\noindent{\it Proof of Proposition \ref{lem2}:} We define $ v_{\alpha}:= \ma^{-\frac{n-2}{2}} u_{\alpha}$. Equation \eqref{m2} rewrites 
\begin{equation}\label{m48}
\left\{\begin{array}{ll}
\Delta_g  v_{\alpha} + a_{\alpha}   v_{\alpha}=\lambda_{\alpha}\ma^{\frac{n-2}{2}(\crit -2)}\frac{ v_{\alpha}^{\crit-1}}{d_g(x,x_0)^s} &\hbox{ in } M \backslash \{x_0\}, \\
v_{\alpha}\geq 0  &\hbox{ in } M \backslash \{x_0\}.
\end{array}\right.
\end{equation}
\noindent We fix $y\in M$ such that $y\neq x_0$. We choose $\delta^{\prime}\in (0,\delta)$ such that $d_g(y,x_0)> \delta^{\prime}$. Let $G_{\alpha}$ be the Green's function of $\Delta_g+a_{\alpha}$. Green's representation formula yields,
\begin{eqnarray*}
v_{\alpha}(y)&=& \ma^{-\frac{n-2}{2}}\ll_{\alpha}\int_M G_{\alpha}(y,x)\frac{\ua^{\crit -1}}{d_g(x,x_0)^s} dv_g\\
&=& \ma^{-\frac{n-2}{2}}\ll_{\alpha}\left( \int_{B_{\delta^{\prime}}(x_0)} G_{\alpha}(y,x)\frac{\ua^{\crit -1}}{d_g(x,x_0)^s} dv_g+\int_{M\backslash B_{\delta^{\prime}}(x_0)} G_{\alpha}(y,x)\frac{\ua^{\crit -1}}{d_g(x,x_0)^s} dv_g\right).
\end{eqnarray*}
On the other hand, since $d_g(x,y)\geq \frac{\delta^{\prime}}{2}$ in the second integral,  using the estimation of $G_{\alpha}$ (see \eqref{m19}) and Theorem \ref{prop1}, we get 
\begin{eqnarray*}
\int_{M\backslash B_{\delta^{\prime}}(x_0)} G_{\alpha}(y,x)\frac{\ua^{\crit -1}}{d_g(x,x_0)^s} dv_g&\leq& C \, \frac{\ma^{\frac{n-2}{2}(\crit -1)}}{{\delta^{\prime}}^{s+(n-2)(\crit -1)}}\int_{M\backslash B_{\delta^{\prime}}(x_0)} d_g(x,y)^{2-n}\, dv_g\\
&\leq&C_{\delta^{\prime}}\, \ma^{\frac{n-2}{2}(\crit -1)}Vol_g(M),
\end{eqnarray*}
we obtain that 
\begin{eqnarray*}
	v_{\alpha}(y)
	&=& \ma^{-\frac{n-2}{2}}\ll_{\alpha} \int_{B_{\delta^{\prime}}(x_0)} G_{\alpha}(y,x)\frac{\ua^{\crit -1}}{d_g(x,x_0)^s} dv_g+O(\ma^{2-s}) \bb{ as } \alpha \to +\infty\\
	&=&\ll_{\alpha} \int_{B_{\delta^{\prime}\ma^{-1}}(0)} G_{\alpha}(y,\exp_{x_0}(\ma X))\frac{\utza^{\crit -1}}{|X|^s} dv_{\tilde{g}_{\alpha}}+O(\ma^{2-s})\bb{ as } \alpha \to +\infty.
\end{eqnarray*}
Thanks again to Step \ref{step00}, \eqref{m1}, the pointwise control \eqref{control:uta} and Lebesgue's Convergence Theorem, we get
\begin{eqnarray}\label{m50}
	\lim_{\alpha\to +\infty}v_{\alpha}(y)= d_n\, G(y,x_0),
\end{eqnarray}
where $d_n$ is defined in \eqref{dn}. The definition of $v_{\alpha}$  and the estimates \eqref{ineq:up} yields, 
\begin{eqnarray*}
	v_{\alpha}(x)\leq c \,d_g(x,x_0)^{2-n} \bb{ for all } x\in M \bb{ and } \alpha\in \nn.
\end{eqnarray*}
\noindent Then, $v_{\alpha}$ is bounded in $L^{\infty}_{loc}(M\backslash \{x_0\}) $. It then follows from \eqref{m48}, \eqref{m50} and elliptic theory that the limit \eqref{m50} in $C^{2}_{loc}(M\backslash \{x_0\})$. This proves Proposition \ref{lem2}.\qed

\section{Direct consequences of Theorem \ref{prop1}}\label{sec:c0theory}

\begin{prop}\label{prop:low} Let $(\ua)_\alpha$ be as in Theorem \ref{prop1}. Let  $(y_{\alpha})_\alpha\in M$ be  such that $y_{\alpha}\to y_0$ as $\alpha\to +\infty$. Then

\begin{equation*}
\lim_{\alpha\to +\infty}\left(\frac{ \ma^{2-s}+\frac{d_g(y_{\alpha},x_0)^{2-s}}{K^{2-s}}}{\ma^{\frac{2-s}{2}}}\right)^{\frac{n-2}{2-s}}\ua(y_{\alpha})=\left\{\begin{array}{cl}
1 & \hbox{ if }y_0=x_0,\\
d_n\left( \frac{d_g(y_0,x_0)}{K}\right)  ^{n-2} G_{x_0}(y_0)& \hbox{ if }y_0\neq x_0,
\end{array}\right. 
\end{equation*}
	where,
	\begin{equation*}
	K^{2-s}=(n-2)(n-s)\mu_s(\rr^n)^{-1} \bb{ and }	d_n:=\mu_s(\rr^n)\int_{\rr^n} \frac{\ut^{\crit -1}}{|X|^s}\, dX,
	\end{equation*}
	and $G_{x_0}$ is the Green's function for $\Delta_g+a_{\infty}$ on $M$ at $x_0$.
\end{prop}
As a consequence, we get that
\begin{coro}\label{prop2ineq} Let $(\ua)_\alpha$ be as in Theorem \ref{prop1}. Then there exists $C>1$ such that 
	\begin{eqnarray*}
	\frac{1}{C}\, \frac{\ma^{\frac{n-2}{2}}}{\left( \ma^{2-s}+\frac{d_g(x,x_0)^{2-s}}{K^{2-s}}\right)^{\frac{n-2}{2-s}} }\leq u_{\alpha}(x)\leq  C\, \frac{\ma^{\frac{n-2}{2}}}{\left( \ma^{2-s}+\frac{d_g(x,x_0)^{2-s}}{K^{2-s}}\right)^{\frac{n-2}{2-s}} }.
	\end{eqnarray*}
\end{coro}

\medskip\noindent{\it Proof of Proposition \ref{prop:low}:} We recall $ \utza(X):= \ma^{\frac{n-2}{2}} u_{\alpha}(\exp_{x_0}(\ma X))$ and satisfies \eqref{eq:utza}. 
It follows from \eqref{convergence} in Step \ref{step00} that $\lim_{\alpha\to +\infty} \tilde{u}_{\alpha}= \tilde{u}$ in  $C^2_{loc}(\rr^n\backslash\{0\})\cap C^{0}_{loc}(\rr^n)$, where $\tilde{u}$ is as in Step \ref{step00} and satisfies \eqref{eq:tu}.  Let $G_{\alpha}$ be the Green's function of $\Delta_g+a_{\alpha}$. We fix $\delta'>0$. As in the proof of Proposition \ref{lem2}, we have that
\begin{eqnarray}\label{conv0}
u_{\alpha}(y_{\alpha})
&=& \ll_{\alpha} \int_{B_{\delta^{\prime}}(x_0)} G_{\alpha}(y_{\alpha},x)\frac{\ua^{\crit -1}}{d_g(x,x_0)^s} dv_g+o\left(\ma^{\frac{n-2}{2}}\right) \bb{ as } \alpha \to +\infty.
\end{eqnarray}

\noindent {\bf Case 1:} We first assume that $\lim_{\alpha\to +\infty}y_\alpha=y_0\neq x_0$. The result is a direct consequence of \eqref{lim:u:G}.

\medskip\noindent {\bf Case 2:} We assume that $\lim_{\alpha\to +\infty}y_\alpha=x_0$.

\smallskip\noindent {\bf Case 2.1:} We assume that there exists $L\in\rr$ such that
\begin{equation}\label{convL}
\frac{d_g(y_{\alpha},x_0)}{\ma}\to L \in\rr\bb{ as } \alpha \to +\infty.
\end{equation} 
We let $Y_{\alpha}\in\rn$ be such that $y_{\alpha}=\exp_{x_0}(\ma Y_{\alpha})$. It follows from \eqref{convL} that
	\begin{eqnarray}\label{convYalpha}
	|Y_{\alpha}|\to L \bb{ as } \alpha \to +\infty.
	\end{eqnarray}
We have that
$$d_{g}(y_{\alpha},x_0)^{n-2}\ma^{-\frac{n-2}{2}}u_{\alpha}(y_{\alpha})= \left( \frac{d_g(y_{\alpha},x_0)}{\ma}\right)^{n-2}\utza(Y_{\alpha}).$$
It then follows from the convergence \eqref{convergence}, \eqref{convL} and \eqref{convYalpha} that

\begin{eqnarray}
	\lim_{\alpha\to +\infty}d_{g}(y_{\alpha},x_0)^{n-2}\ma^{-\frac{n-2}{2}}u_{\alpha}(y_{\alpha})=\left( \frac{L^{2-s}}{1+\frac{L^{2-s}}{K^{2-s}}}\right)^{\frac{n-2}{2-s}}.\label{lim:L}
\end{eqnarray}

\medskip\noindent {\bf Case 2.2:} We assume that  
\begin{equation}\label{conyalpha}
y_\alpha\to x_0\hbox{ and }\frac{d_g(y_{\alpha},x_0)}{\ma}\to +\infty \bb{ as } \alpha \to +\infty.
\end{equation}
Coming back to \eqref{conv0}, we have as $\alpha\to +\infty$ that,
	\begin{eqnarray}
		d_{g}(y_{\alpha},x_0)^{n-2}\ma^{-\frac{n-2}{2}}u_{\alpha}(y_{\alpha})&=&d_g(y_{\alpha},x_0)^{n-2} \ma^{-\frac{n-2}{2}}\ll_{\alpha}\left(  \int_{D_{1,\alpha}} G_{\alpha}(y_{\alpha},x)\frac{\ua^{\crit -1}}{d_g(x,x_0)^s}\,  dv_g\right. \nonumber\\
		&&\left. + \int_{D_{2,\alpha}} G_{\alpha}(y_{\alpha},x)\frac{\ua^{\crit -1}}{d_g(x,x_0)^s}\,  dv_g\right)+O\left( \ma^{\frac{n-2}{2}(\crit -2)}\right),\label{eq:Dalpha}
	\end{eqnarray} 
	with, $$ D_{1,\alpha}:=\left\lbrace x\in B_{\delta}(x_0); d_{g}(y_{\alpha},x)\geq \frac{1}{2}d_g(y_{\alpha},x_0) \right\rbrace \bb{ and } D_{2,\alpha}:= B_{\delta}(x_0)\backslash D_{1,\alpha}.$$ 
	With a change of variable, we get 
	\begin{equation}\label{est:G:u:1}
		\ma^{-\frac{n-2}{2}}  \int_{D_{1,\alpha}} G_{\alpha}(y_{\alpha},x)\frac{\ua^{\crit -1}}{d_g(x,x_0)^s}\,  dv_g= \int_{D^{\prime}_{1,\alpha}} G_{\alpha}(y_{\alpha},\exp_{x_0}(\ma X))\frac{\utza^{\crit -1}}{|X|^s}\,  dv_{\tilde{g}_{\alpha}},
	\end{equation}
	where $D^{\prime}_{1,\alpha}=\ma^{-1}\exp_{x_0}(D_{1,\alpha})$. For $R>0$, we take  $X\in B_R(0)$ and  $z_{\alpha}:= \exp_{x_0}(\ma X)$,  by \eqref{conyalpha} we have that 
	\begin{eqnarray}\label{convyzalpha}
	\frac{d_g(y_{\alpha},z_{\alpha})}{\ma} \to +\infty \bb{ as } \alpha \to +\infty.
	\end{eqnarray}
	Writing, 
	$$ d_g(y_{\alpha},z_{\alpha})-d_g(z_{\alpha},x_0)\leq d_g(y_{\alpha},x_0)\leq d_g(y_{\alpha},z_{\alpha})+d_g(z_{\alpha},x_0),$$
	and nothing that $d_g(z_{\alpha},x_0)=\ma |X|$, we obtain that 
	\begin{eqnarray*}
		1-|X|\frac{\ma}{d_g(y_{\alpha},z_{\alpha})}\leq \frac{d_g(y_{\alpha},x_0)}{d_g(y_{\alpha},z_{\alpha})}\leq 1+|X|\frac{\ma}{d_g(y_{\alpha},z_{\alpha})},
	\end{eqnarray*}
	therefore, with \eqref{convyzalpha}, we get  
	\begin{equation*}
	\lim_{\alpha\to +\infty} \frac{d_g(y_{\alpha},x_0)}{d_g(y_{\alpha},z_{\alpha})}=1.
	\end{equation*}
	Therefore for all $R>0$, we have that $B_R(0)\subset D'_{1,\alpha}$ for $\alpha>0$ large enough. Moreover, since $d_g(y_{\alpha},z_{\alpha})\to 0$ as $\alpha\to +\infty$ and by  Proposition $12$ in Robert \cite{RG}, we have 
	\begin{eqnarray*}
		\lim_{\alpha\to +\infty	} d_g(y_{\alpha},x_0)^{n-2}G_{\alpha}(y_{\alpha},z_{\alpha})&=& \left( \frac{d_g(y_{\alpha},x_0)}{d_g(y_{\alpha},z_{\alpha})} \right)^{n-2} d_g(y_{\alpha},z_{\alpha})^{n-2}G_{\alpha}(y_{\alpha},z_{\alpha})\\
		&=& \frac{1}{(n-2)\omega_{n-1}},
	\end{eqnarray*}
	where $\omega_{n-1}$ is the volume of the unit $(n - 1)$-sphere. It then follows from \eqref{est:G:u:1}, \eqref{m19} the pointwise control \eqref{control:uta} and Lebesgue's Convergence Theorem that
	\begin{eqnarray}
	\lim_{\alpha \to +\infty}d_g(y_{\alpha},x_0)^{n-2}\ma^{-\frac{n-2}{2}}\ll_{\alpha}  \int_{D_{1,\alpha}} G_{\alpha}(y_{\alpha},x)\frac{\ua^{\crit -1}}{d_g(x,x_0)^s}\,  dv_g\nonumber\\
	=\mu_s(\rr^n)\frac{1}{(n-2)\omega_{n-1}}\int_{\rr^n} \frac{\ut^{\crit-1}}{|X|^s}\,  dX\label{convstep20}
	\end{eqnarray}
 Now, going back to the definition of $\tilde{u}$ (see \eqref{eq:utilde}) and with a change of variable,
	\begin{eqnarray}
		\int_{\rr^n} \frac{\ut^{\crit-1}}{|X|^s}\,  dX
		&=&K^{n-s} \int_{\rr^n} |X|^{-s}\left(1+|X|^{2-s} \right)^{-\frac{n-2}{2-s}(\crit-1)}\, dX\nonumber\\
		&=& K^{n-s}\omega_{n-1} \int_{0}^{+\infty} \frac{r^{n-s-1}}{\left(1+r^{2-s} \right)^{\frac{n-2}{2-s}(\crit-1)}}\, dr\nonumber\\
			&=& K^{n-s}\frac{\omega_{n-1}}{2-s} \int_{0}^{+\infty} \frac{t^{\frac{n-s}{2-s}-1}}{\left(1+t \right)^{\frac{n-2s+2}{2-s}}}\, dr\nonumber\\
		&=& K^{n-s}\frac{\omega_{n-1}}{2-s} \frac{\Gamma(\frac{n-s}{2-s})\Gamma(1)}{\Gamma(\frac{n-s}{2-s}+1)}=K^{n-s}\frac{\omega_{n-1}}{n-s}.\label{eq:fonctiongamma}
	\end{eqnarray}
	Since $\mu_s(\rr^n)=K^{s-2}(n-2)(n-s)$ and by \eqref{convstep20}, we get
	\begin{equation}\label{convstep21}
	\lim_{\alpha \to +\infty}d_g(y_{\alpha},x_0)^{n-2}\ma^{-\frac{n-2}{2}}\ll_{\alpha}  \int_{D_{1,\alpha}} G_{\alpha}(y_{\alpha},x)\frac{\ua^{\crit -1}}{d_g(x,x_0)^s}\,  dv_g= K^{n-2}.
	\end{equation} 
	Note that $d_g(x,x_0)\geq\frac{1}{2} d_g(y_\alpha,x_0)$ for all $x\in D_{2,\alpha}$. Then, it follows from  \eqref{ineq:up}, \eqref{m19} and \eqref{conyalpha} that
		\begin{eqnarray}
	&&\ma^{-\frac{n-2}{2}} d_g(y_{\alpha},x_0)^{n-2} \int_{D_{2,\alpha}} G_{\alpha}(y_{\alpha},x)\frac{\ua^{\crit -1}}{d_g(x,x_0)^s}\,  dv_g\nonumber\\
	&&\leq C\,	\left(\frac{\ma}{d_g(y_{\alpha},x_0)} \right)^{2-s}\frac{1}{d_g(y_{\alpha},x_0)^{2}}\int_{D_{2,\alpha}}  d_g(y_{\alpha},x)^{2-n}\, dv_g\nonumber\\
	&&\leq C\, \left(\frac{\ma}{d_g(y_{\alpha},x_0)} \right)^{2-s}=o(1). \label{convstep22} 
	\end{eqnarray}
	Combining \eqref{eq:Dalpha}, \eqref{convstep21} and  \eqref{convstep22}, we write that
\begin{equation}\label{convstep200}
\lim_{\alpha \to +\infty} d_{g}(y_{\alpha},x_0)^{n-2}\ma^{-\frac{n-2}{2}}u_{\alpha}(y_{\alpha})=K^{n-2}.
\end{equation}
Proposition \ref{prop:low} is a direct consequence of \eqref{lim:u:G}, \eqref{lim:L} and \eqref{convstep200}.\qed 

\medskip\noindent{\it Proof of Corollary \ref{prop2ineq}:} We define  
\begin{equation*}
v_{\alpha}(x):=\left( \frac{\ma^{2-s}+\frac{d_g(x,x_0)^{2-s}}{K^{2-s}}}{\ma^{\frac{2-s}{2}}}\right)^{\frac{n-2}{2-s}}u_{\alpha}(x)
\end{equation*}
for all $x\in M$ and $\alpha\in\nn$. We let $(y_{\alpha})_\alpha\in M$ be such that $v_{\alpha}(y_{\alpha})=\min_{x\in M} v_{\alpha}(x)$ for all $\alpha\in\nn$. Since $G_\alpha>0$, it follows from Proposition \ref{prop:low} that there exists $c_0>0$ such that $v_{\alpha}(y_{\alpha}) \geq c_0$ for all $\alpha\in\nn$. This yields the lower bound of Corollary \ref{prop2ineq}. The upper bound is \eqref{ineq:up}. This proves Corollary  \ref{prop2ineq}.\qed
\section{Pohozaev identity and proof of Theorem \ref{th1}}\label{sec:pohozaev}
We let $(u_{\alpha})_{\alpha}\in H^2_1(M)$, $(\ll_\alpha)_{\alpha}\in \rr$ and  $(a_\alpha)_{\alpha}, a_\infty$ in $C^1(M)$ be such that \eqref{m0} to \eqref{m00} hold. In the sequel, we fix $\delta \in (0, \frac{i_g(M)}{2})$ where $i_g(M)>0$ is the injectivity radius of $(M,g)$. We consider the following function,
\begin{equation*}
\uha(X):= u_{\alpha}(\exp_{x_{0}}(X))\bb{ for all } X\in B_{\delta}(0)\subset \rr^n,
\end{equation*}
where $\exp_{x_{0}}:B_{\delta}(0)\to B_{\delta}(x_{0})\subset M$ is the exponential map at $x_{0}$.
We define also the metric $$\hat{g}(X):=\left( \exp^\star_{x_{0}} g\right) (X) \bb{ on } \rr^n.$$
It then follows from \eqref{m2} that 
\begin{equation}\label{eq:uha}
	\Delta_{\hat{g}}  \uha + \hat{a}_{\alpha}   \uha=\lambda_{\alpha}\frac{ \uha^{\crit-1}}{|X|^s} \bb{ weakly in }  B_{\delta}(0),
	\end{equation}
	where $\hat{a}_{\alpha}= a_{\alpha}(\exp_{x_0}( X))$. For $l\geq 1$, the Pohozaev identity writes (see for instance Ghoussoub-Robert \cite{GRCalcVar})
	
\begin{eqnarray*}
&&\int_{B_{\delta}(0)} \left( X^l \partial_l \uha+\frac{n-2}{2}\uha \right)\left(\Delta_{Eucl} \uha-\lambda_{\alpha} \frac{\uha^{\crit-1}}{|X|^s}\right) \, dX\\
&&=\int_{\partial B_{\delta}(0)} \left( X, \nu\right) \left( \frac{|\nabla \uha |^2}{2} -\frac{\lambda_{\alpha}}{\crit}\frac{\uha^{\crit}}{|X|^s}\right)- \left(X^l \partial_l \uha+\frac{n-2}{2}\uha \right) \partial_{\nu}\uha \, d\sigma,
\end{eqnarray*}
where $\nu(X)$ is the outer normal vector of $B_{\delta}(0)$ at $X\in \partial B_{\delta}(0)$, that is $\nu(X)=\frac{X}{|X|}.$ With \eqref{eq:uha}, the Pohozaev identity writes
\begin{equation}\label{bcd}
C_\alpha+D_\alpha=B_\alpha,
\end{equation}
with
$$ B_{\alpha
}:=\int_{\partial B_{\delta}(0)} \left( X, \nu\right) \left( \frac{|\nabla \uha |^2}{2} -\frac{\lambda_{\alpha}}{\crit}\frac{\uha^{\crit}}{|X|^s}\right)- \left(X^l \partial_l \uha+\frac{n-2}{2}\uha \right) \partial_{\nu}\uha \, d\sigma.$$
$$C_{\alpha}:= -\int_{B_{\delta}(0)} \left( X^l \partial_l \uha+\frac{n-2}{2}\uha \right)  \hat{a}_{\alpha}\uha   \, dX,$$
and, 
$$D_{\alpha}:= -\int_{B_{\delta}(0)} \left( X^l \partial_l \uha+\frac{n-2}{2}\uha \right)  \left(  \Delta_{\hat{g}}\uha-\Delta_{Eucl}\uha  \right) \, dX.$$
We are going to estimate these terms separately.

\begin{step}\label{step:5} We claim that
\begin{equation}\label{balpha}
	\lim_{\alpha\to +\infty}\frac{B_\alpha}{\ma^{n-2}}=d_n^2\, \left( \int_{\partial B_{\delta}(0)} \frac{\delta}{2}|\nabla \hat{G}_{x_0} |^2  -\frac{1}{\delta}\left( \langle X, \nabla \hat{G}_{x_0}\rangle^2 +\frac{n-2}{2}\langle X, \nabla \hat{G}_{x_0}\rangle \hat{G}_{x_0} \right) \, d\sigma \right) ,
	\end{equation}
as $\alpha\to +\infty$, where $d_n$ is defined in \eqref{dn}, and  $\hat{G}_{x_{0}}(X)=G(x_0, \exp_{x_{0}}(X))$.
\end{step}
\smallskip\noindent{\it Proof of Claim \ref{step:5}:} It follows from the definition of $\hat{u}_{\alpha}$ that
\begin{eqnarray*}
\ma^{-\left( n-2\right) } B_{\alpha}&=& \int_{\partial B_{\delta}(0)} \left( X, \nu\right) \left(\frac{|\ma^{-\frac{n-2}{2}}\nabla \uha |^2}{2}  -\frac{\lambda_{\alpha}}{\crit} \ma^{2-s}\frac{\left( \ma^{-\frac{n-2}{2}}\uha\right) ^{\crit}}{|X|^s}\right)\\
&&- \left(X^l \ma^{-\frac{n-2}{2}}\partial_l \uha+\frac{n-2}{2}\ma^{-\frac{n-2}{2}}\uha \right) \ma^{-\frac{n-2}{2}}\partial_{\nu}\uha \, d\sigma.
\end{eqnarray*}
Since $\ma \to 0$ as $\alpha\to +\infty$, the convergence of Proposition \ref{lem2} yields \eqref{balpha}. This proves the claim.\qed

\medskip\noindent In this section, we will extensively use the following consequences of the pointwise estimates \eqref{ineq:up} and \eqref{m41}:
\begin{eqnarray}
\uha(X)&\leq&  C\left(\frac{\ma}{\ma^2+|X|^2}\right)^{\frac{n-2}{2}} \hbox{ in }B_{\delta}(0),\label{ineq:hu}\\
|\nabla\uha|(X)&\leq & C\frac{\ma^{\frac{n-2}{2}}}{\left(\ma^2+|X|^2\right)^{\frac{n-1}{2}}} \hbox{ in }B_{\delta}(0)\setminus B_{R\ma}(0),\label{ineq:nablahu}
\end{eqnarray}
and
\begin{eqnarray}
\utza(X)&\leq&  C\left(\frac{1}{1+|X|^2}\right)^{\frac{n-2}{2}} \hbox{ in }B_{\delta\ma^{-1}}(0),\label{ineq:tu}\\
|\nabla\utza|(X)&\leq & C\frac{1}{\left(1+|X|^2\right)^{\frac{n-1}{2}}} \hbox{ in }B_{\delta\ma^{-1}}(0)\setminus B_R(0),\label{ineq:nablatu}
\end{eqnarray}
where \begin{equation}\label{eq:uta}
\uta(X)= \ma^{\frac{n-2}{2}}u_\alpha(\exp_{x_0}(\ma X)) \bb{ for all } X\in B_{\delta\ma^{-1}}(0)\subset \rr^n.
\end{equation}

  \begin{step}\label{stepcalpha} We claim that, as $\alpha\to +\infty$, 
	\begin{equation*}\label{eq:20}
	C_{\alpha}=\left\{\begin{array}{ll}  
	\ma^2\ln\left(\frac{1}{\ma}\right)\left(  \omega_3K^4a_{\infty}(x_0)+o(1)\right) &\hbox{ if }n=4, \\\\
	\ma^2 \left(a_{\infty}(x_0)\int_{\rr^n} \tilde{u}^2 \, dX+o(1) \right) &\hbox{ if } n\geq 5,
	\end{array}\right.
	\end{equation*}
	where $K$ is defined in \eqref{defK}.
\end{step}
\noindent{\it Proof of Step \ref{stepcalpha}:} Using the definition of $C_{\alpha}$ and integrating by parts, we get 
\begin{eqnarray}
C_{\alpha}&&=  -\int_{B_{\delta}(0)} \left( X^l \hat{a}_{\alpha}\partial_l\left( \frac{\uha^2}{2} \right) +\frac{n-2}{2}\hat{a}_{\alpha}\uha^2 \right)    \, dX\nonumber\\
&&=
-\int_{B_{\delta}(0)} \left(-\frac{n}{2} \hat{a}_{\alpha} \uha^2-X^l\partial_l \hat{a}_{\alpha} \frac{\uha^2}{2} +\frac{n-2}{2}\hat{a}_{\alpha}\uha^2 \right)    \, dX\nonumber\\
&&-\frac{1}{2}\int_{\partial B_{\delta}(0)}\left(X, \nu \right) \hat{a}_{\alpha}\uha^2    \, d\sigma \nonumber\\
&&= \int_{B_{\delta}(0)} \left( \hat{a}_{\alpha} +\frac{X^l\partial_l \hat{a}_{\alpha}}{2}   \right) \uha^2   \, dX-\frac{1}{2}\int_{\partial B_{\delta}(0)}\left(X, \nu \right) \hat{a}_{\alpha}\uha^2    \, d\sigma.\label{calpha}
\end{eqnarray}
With \eqref{ineq:hu}, we can write that 
\begin{eqnarray*}
	\left| \int_{\partial B_{\delta}(0)}\left(X, \nu \right) \hat{a}_{\alpha}\uha^2    \, d\sigma\right|
	&\leq&C(\delta)\ma^{n-2}\int_{\partial B_{\delta}(0)}\frac{1}{|X|^{2(n-2)} } \, d\sigma,
\end{eqnarray*}
then, 
\begin{equation*}
\int_{\partial B_{\delta}(0)}\left(X, \nu \right) \hat{a}_{\alpha}\uha^2    \, d\sigma=O\left( \ma^{n-2}\right).
\end{equation*}
Moreover, with \eqref{calpha} we have 
\begin{equation*}
C_{\alpha}=\int_{B_{\delta}(0)} \left( \hat{a}_{\alpha} +\frac{X^l\partial_l \hat{a}_{\alpha}}{2}   \right) \uha^2   \, dX+O(\ma^{n-2}) \bb{ as } \alpha\to +\infty.
\end{equation*}
We now define
	\begin{equation}\label{varphialpha}
\varphi_{\alpha}(X):=\hat{a}_{\alpha} +\frac{X^l\partial_l \hat{a}_{\alpha}}{2}.
	\end{equation}
\noindent We distinguish three cases:

\noindent {\bf Case 1:} If $n\geq 5$, 
with a change of variable $X=\ma Y$, we get that 
\begin{eqnarray*}
	\int_{B_{\delta}(0)} \varphi_{\alpha}(X)\uha^2   \, dX&=& \ma^2\int_{B_{\delta\ma^{-1}}(0)} \varphi_{\alpha}(\ma X)\utza^2   \, dX,
\end{eqnarray*}
where $\uta$ is defined in \eqref{eq:uta}. Since $\ma\to 0$ as $\alpha\to +\infty$, we use \eqref{varphialpha} and  \eqref{m0}
$$ \lim_{\alpha\to+\infty}\varphi_{\alpha}(\ma X)=a_{\infty}(x_0).$$ 
Since $n\geq 5$, we have that $X\mapsto  \left( 1+|X|^2\right) ^{-\frac{n-2}{2}}\in L^{2}(\rr^n)$. Therefore, with the pointwise control \eqref{ineq:tu}, Lebesgue's dominated convergence theorem and Step \ref{step00} yield Step \ref{stepcalpha} when $n\geq 5$.

\medskip\noindent {\bf Case 2:} If $n=4$, we have that

\begin{eqnarray}
&&K^{-4}\int_{B_{\delta\ma^{-1}}(0)}\frac{1}{\left(1+\left( \frac{|X|}{K}\right) ^{2-s} \right)^{\frac{4}{2-s}} }\, dX =\int_{B_{\delta\ma^{-1}}(0)}\frac{1}{\left(K^{2-s}+|X|^{2-s} \right)^{\frac{4}{2-s}}}\, dX\nonumber\\
&&= \omega_3\int_{1}^{\delta\ma^{-1}}\frac{r^3}{\left(K^{2-s}+r^{2-s} \right)^{\frac{4}{2-s}} }\, dr+O(1)\nonumber\\
&&=\omega_3\int_{1}^{\delta\ma^{-1}} \frac{1 }{r }   \, dr+\int_{1}^{\delta\ma^{-1}}r^3 \left[  \frac{ 1-\left(1+\frac{1}{r^2} \right)^2 }{\left( 1+r^2\right)^{2} }\right]   \, dr+O(1)\nonumber\\
&&=\omega_3\ln\left(\frac{\delta}{\ma}\right)+O\left( \int_{1}^{\delta\ma^{-1}}\frac{1 }{r^3 } \, dr\right)+ O(1) \nonumber\\
&&=\omega_3\ln\left(\frac{\delta}{\ma}\right)+O(1).\label{convergen4}
\end{eqnarray}
Therefore, it follows from Proposition \ref{prop2ineq}, for any $\epsilon >0$, there exists $\delta_{\epsilon}>0$ such
that, up to a subsequence, for any $\alpha$ and any $X\in B_{\delta_{\epsilon}}(0)$, 
\begin{equation}\label{convergen40}
\frac{1}{1+\epsilon} \frac{1}{\left( 1+\left( \frac{|X|}{K}\right) ^{2-s}\right)^{\frac{2}{2-s}}}\leq 	\utza(X)\leq  \left( 1+\epsilon\right) \,  \frac{1}{\left( 1+\left( \frac{|X|}{K}\right) ^{2-s}\right)^{\frac{2}{2-s}}}.
\end{equation}
Combining the last equation and \eqref{convergen4}, by letting $\alpha\to +\infty$ and then $\epsilon\to 0$, we obtain that 
\begin{equation*}
\lim_{\alpha\to +\infty}\frac{1}{\ln(\frac{1}{\ma})} \int_{B_{\delta\ma^{-1}}(0)}\varphi_{\alpha}(X)\uha^2   \, dX =\omega_3K^4a_{\infty}(x_0).
\end{equation*}
This yields Step \ref{stepcalpha} for $n=4$. These two cases yield Step \ref{stepcalpha}.\qed

\smallskip\noindent We shall make frequent use of the following Lemma in dimension $n=4$.
	\begin{lem}\label{eq:Conv:n=4}
	For $i,j,\beta_1,\beta_2 \geq 1$, and $n=4$, we claim that
	\begin{equation}\label{conv:attends:nabla}
	\lim_{\alpha\to +\infty}	\frac{\int_{B_{\delta\ma^{-1}}(0)} X^{\beta_1}X^{\beta_2} \partial_i \utza \partial_j \utza \, dX }{ \ln(\frac{1}{\ma})}=\left( (n-2)K^{n-2}\right)^2\int_{S^{n-1}}\sigma^i \sigma^j \sigma^{\beta_1}\sigma^{\beta_2} \, d\sigma, 
	\end{equation}
	where $\uta$ is defined in \eqref{eq:uta}, and  
	\begin{equation}\label{eq:K}
	K^{2-s}=(n-2)(n-s)\mu_s(\rr^n)^{-1}.
	\end{equation}
\end{lem}
\noindent\textit{Proof of the Lemma} \ref{eq:Conv:n=4}: We divide the proof into several steps.
\begin{step}\label{procedure:conv:n=4}
	We fix a family of parameters $(\beta_\alpha)\in (0,\infty)$ such that 
	\begin{equation}\label{eq:convbetaalpha}
	\lim_{\alpha\to +\infty}\beta_\alpha=0 \bb{ and } \lim_{\alpha\to +\infty}\frac{\ma}{\beta_\alpha}=0.
	\end{equation}
	Then, for all $X\in \rr^n\backslash \{0\}$, we have that 
	\begin{equation*}
	\lim_{\alpha\to +\infty} \frac{\beta_{\alpha}^{n-2}}{\ma^{\frac{n-2}{2}}}u_{\alpha}(\exp_{x_0}(\beta_{\alpha} X))=K^{n-2} |X|^{2-n}.
	\end{equation*}
	Moreover, this limit holds in $C^2_{loc}(\rr^n\backslash\{0\})$.
\end{step}
\noindent\textit{Proof of the Step} \ref{procedure:conv:n=4}: First, we define 
$$ w_{\alpha}(X):= \frac{\beta_{\alpha}^{n-2}}{\ma^{\frac{n-2}{2}}}u_\alpha(\exp_{x_0}(\beta_{\alpha}X)) \bb{ for all } X\in \rr^n\cap \beta_\alpha^{-1}U.$$
\medskip\noindent{\it Step \ref{procedure:conv:n=4}.1:} We claim that there exists $w\in C^2(\rr^n\backslash\{0\})$ 
\begin{equation*}
\lim_{\alpha\to +\infty} w_\alpha=w \bb{ in } C^2_{loc}(\rr^n\backslash\{0\}).
\end{equation*} 
Therefore, there exists $\Lambda\geq 0$ such that 
$$w(X)=\Lambda |X|^{2-n} \bb{ for all } X\in \rr^n.$$
\medskip\noindent{\it Proof of the Step \ref{procedure:conv:n=4}.1:} Since $u_\alpha$ satisfies \eqref{m2}, with the definition of $w_\alpha$, we have that
\begin{equation}\label{eq:hardywalpha}
\left\{\begin{array}{ll}
\Delta_g  w_{\alpha} + \beta_{\alpha}^2a_{\alpha}(\exp_{x_0}(\beta_{\alpha}X))   w_{\alpha}=\lambda_{\alpha}\, \left( \frac{\ma^{\frac{n-2}{2}}}{\beta_{\alpha}^{n-2}}\right)^{\crit-2}\beta_{\alpha}^{2-s} \frac{ w_{\alpha}^{\crit-1}}{|X|^s}&\hbox{ in }\rr^n\cap \beta_\alpha^{-1}U, \\
w_\alpha\geq 0  &\hbox{ in } \rr^n\cap \beta_\alpha^{-1}U.
\end{array}\right.
\end{equation}
Since $\crit >2$, we get with \eqref{eq:convbetaalpha} that
\begin{eqnarray*}
	\left( \frac{\ma^{\frac{n-2}{2}}}{\beta_{\alpha}^{n-2}}\right)^{\crit-2}\beta_{\alpha}^{2-s} &=&\left( \frac{\ma}{\beta_{\alpha}}\right)^{{\frac{n-2}{2}}(\crit-2)}=o(1) \bb{ as } \alpha\to +\infty.
\end{eqnarray*}
From the pointise control \eqref{ineq:up} that there existe $C>0$ such that 
\begin{eqnarray*}
	0< w_\alpha(X)\leq C \, |X|^{2-n} \bb{ for all } X\in \rr^n\cap \beta_\alpha^{-1}U.
\end{eqnarray*}
It follows from standard elliptic theory that there exists $w\in C^2(\rr^n\backslash\{0\})$ such that 
\begin{equation*}
\lim_{\alpha\to +\infty} w_\alpha=w \bb{ in } C^2_{loc}(\rr^n\backslash\{0\}).
\end{equation*}
Passing to the limit as $\alpha\to +\infty$ in \eqref{eq:hardywalpha}, 
\begin{equation*}
\left\{\begin{array}{ll}
\Delta_{Eucl}  w=0&\hbox{ in }\rr^n, \\
0\leq w(X)\leq C \, |X|^{2-n}& \hbox{ in }\rr^n,
\end{array}\right.
\end{equation*}
that there exists $\Lambda\geq 0$ such that 
$$w(X)=\Lambda |X|^{2-n} \bb{ for all } X\in \rr^n.$$ 
This ends Step \ref{procedure:conv:n=4}.1.\qed

\medskip\noindent{\it Step \ref{procedure:conv:n=4}.2:}  We are left with proving that $\Lambda=K^{n-2}$ defined in \eqref{eq:K}. 

\medskip\noindent{\it Proof of the Step \ref{procedure:conv:n=4}.2:} We fix $X\in \rr^n$. Let $G_\alpha$ the Green's function of $\Delta_g+a_\alpha$. Green's representation formula and the defintion of $w_\alpha$ yields,
\begin{eqnarray}
w_\alpha(X)&=&\ll_\alpha\int_M \frac{\beta_{\alpha}^{n-2}}{\ma^{\frac{n-2}{2}}}G_\alpha(\exp_{x_0}(\beta_\alpha X),y)\frac{u_{\alpha}(y)^{\crit-1}}{d_g(y,x_0)^s}\,dv_g\nonumber\\
&=&A_{\alpha}+B_\alpha,\label{eq:Aalpha+Balpha}
\end{eqnarray}
where, 
\begin{eqnarray*}
	&&A_{\alpha}:=\ll_\alpha\,\frac{\beta_{\alpha}^{n-2}}{\ma^{\frac{n-2}{2}}}\int_{ B_{R\ma}(x_0)\backslash B_{\delta\ma}(x_0)  } G_\alpha(\exp_{x_0}(\beta_\alpha X),y)\frac{u_{\alpha}(y)^{\crit-1}}{d_g(y,x_0)^s}\,dv_g,\\
	&&B_\alpha:=\ll_\alpha\,\frac{\beta_{\alpha}^{n-2}}{\ma^{\frac{n-2}{2}}}\int_{M\backslash\left( B_{R\ma}(x_0)\backslash B_{\delta\ma}(x_0)\right) } G_\alpha(\exp_{x_0}(\beta_\alpha X),y)\frac{u_{\alpha}(y)^{\crit-1}}{d_g(y,x_0)^s}\,dv_g.
\end{eqnarray*}
\medskip\noindent{\it Step \ref{procedure:conv:n=4}.2.1:} We claim that  
\begin{eqnarray}\label{eq:Aalpha}
\lim_{R\to +\infty, \delta\to 0}\lim_{\alpha\to +\infty} A_\alpha=\frac{K^{n-2}}{| X|^{n-2}}.
\end{eqnarray}
\medskip\noindent{\it Proof of the Step \ref{procedure:conv:n=4}.2.1:} Taking $y=\exp_{x_0}(\ma Y)$, we write 
\begin{equation}\label{eq:Aalpha1}
A_\alpha=\ll_\alpha\,\int_{ B_{R}(0)\backslash B_{\delta}(0)  } \beta_{\alpha}^{n-2}G_\alpha(x_\alpha,y_\alpha)\frac{\uta^{\crit-1}}{|Y|^s}\,dv_{\tilde{g}_{\alpha}},
\end{equation}
where,
$$x_\alpha:=\exp_{x_0}(\beta_\alpha X) \bb{ and }y_\alpha:=\exp_{x_0}(\ma Y),$$
and  $\uta$ is defined in \eqref{eq:uta} and 
\begin{equation*}
\tilde{g}_{\alpha}(x):=\left( \exp^\star_{x_{0}} g\right) (\mu_{\alpha}X)\hbox{ in }B_{\delta^{-1}\ma}(0).
\end{equation*}
The triangle inequality yields
\begin{eqnarray*}
	d_g(x_\alpha,x_0)-d_g(y_\alpha,x_0)\leq d_g(x_\alpha,y_\alpha)\leq d_g(x_\alpha,x_0)+d_g(y_\alpha,x_0),
\end{eqnarray*}
and since $d_g(x_\alpha,x_0)=\beta_\alpha|X|$ and $d_g(y_\alpha,x_0)=\ma |Y|$, we get that
\begin{eqnarray*}
	|X|-\frac{\ma}{\beta_{\alpha}}|Y|\leq \frac{d_g(x_\alpha,y_\alpha)}{\beta_\alpha}\leq |X|+\frac{\ma}{\beta_{\alpha}}|Y|,
\end{eqnarray*}
therefore, with \eqref{eq:convbetaalpha}, 
$$d_g(x_\alpha,y_\alpha)\to 0 \bb{ and }  \frac{\beta_\alpha}{d_g(x_\alpha,y_\alpha)}\to \frac{1}{|X|},$$
as $\alpha\to +\infty$. Therefore, it follows from  Proposition $12$ in Robert \cite{RG} that 
\begin{eqnarray*}
	\lim_{\alpha\to +\infty}\beta_\alpha^{n-2}	G_\alpha(x_\alpha,y_\alpha)
	&=&\lim_{\alpha\to +\infty}\left( \frac{\beta_\alpha}{d_g(x_\alpha,y_\alpha)}\right)^{n-2}	d_g(x_\alpha,y_\alpha)^{n-2}G_\alpha(x_\alpha,y_\alpha)\\
	&=&\frac{1}{(n-2)w_{n-1}} \frac{1}{| X|^{n-2}} 
\end{eqnarray*} 
uniformly. Therefore, with \eqref{eq:Aalpha1}, applying again Lebesgue's Convergence Theorem and thanks to Step \ref{step00} and the convergence in \eqref{m1}, we infer that  \begin{equation}\label{eq:Aalpha2}
A_\alpha=\frac{\mu_s(\rr^n)}{(n-2)w_{n-1}} \frac{1}{| X|^{n-2}}\int_{ B_{R}(0)\backslash B_{\delta}(0)  } \frac{\ut(Y)^{\crit-1}}{|Y|^s}\,dY+o(1) \bb{ as } \alpha \to + \infty.
\end{equation}
Other, going back to \eqref{eq:fonctiongamma}, we have that  
\begin{eqnarray*}
	\int_{ \rr^n  } \frac{\ut(Y)^{\crit-1}}{|Y|^s}\,dY= K^{n-s}\frac{\omega_{n-1}}{n-s} \bb{ and } \mu_s(\rr^n)=K^{s-2}(n-2)(n-s).
\end{eqnarray*}
Replacing the last equation in \eqref{eq:Aalpha2}, we get \eqref{eq:Aalpha}. This ends Step \ref{procedure:conv:n=4}.2.1.\qed

\medskip\noindent{\it Step \ref{procedure:conv:n=4}.2.2:}
\begin{eqnarray}\label{eq:Balpha}
\lim_{R\to +\infty, \delta\to 0}\lim_{\alpha\to +\infty} B_\alpha=0.
\end{eqnarray}
\medskip\noindent{\it Proof of the Step \ref{procedure:conv:n=4}.2.2:} By the definition of $B_\alpha$ and the estimate's on the Green's function (see \eqref{m19} and Robert \cite{RG}) and $u_\alpha$ see \eqref{ineq:up}, then there exists $C>0$ such that
\begin{eqnarray*}
	B_\alpha
		&\leq&C \, \ll_\alpha\frac{\beta_{\alpha}^{n-2}}{\ma^{-\frac{n-2}{2}(\crit-2)}}\int_{M\backslash\left( B_{R\ma}(x_0)\backslash B_{\delta\ma}(x_0)\right) } \frac{d_g(\exp_{x_0}(\beta_\alpha X),y)^{2-n}}{\left( \ma^2+d_g(y,x_0)^2\right)^{\frac{n-2}{2}(\crit-1)}d_g(y,x_0)^s}\,dv_g,
\end{eqnarray*}
For $R_0>0$, taking $y=\exp_{x_0}(\ma Y)$, we have that
\begin{eqnarray}
	B_\alpha&\leq& C \, \ll_\alpha \int_{ B_{R_0\ma^{-1}}(0)\backslash\left( B_{R}(0)\backslash B_{\delta}(0)\right) } \frac{| X-\frac{\ma}{\beta_\alpha} Y|^{2-n}}{\left(1+|Y|^2\right)^{\frac{n-2}{2}(\crit-1)}|Y|^s}\,dY\nonumber\\
	&=&C(B_{1,\alpha}+B_{2,\alpha}),\label{eq:Balpha12}
\end{eqnarray} 
where 
\begin{eqnarray*}
	B_{1,\alpha}= \int_{  B_{\delta}(0)} J_{\alpha}\,dY \bb{ and } B_{2,\alpha}=\int_{  \rr^n\backslash B_R(0)}  J_{\alpha}\,dY,
\end{eqnarray*}
and
 \begin{equation*}
J_{\alpha}:=\ll_\alpha \frac{| X-\frac{\ma}{\beta_\alpha} Y|^{2-n}}{\left(1+|Y|^2\right)^{\frac{n-2}{2}(\crit-1)}|Y|^s}.
\end{equation*}
 First, we estimate $B_{1,\alpha}$. Since $\frac{\ma}{\beta_\alpha}\to 0$, we have $|X-\frac{\ma}{\beta_\alpha} Y|^{2-n}\to |X|^{2-n}$ uniformly on  $B_{\delta}(0)$ as $\alpha \to +\infty$ and with the convergence of $\ll_\alpha$ in \eqref{m1}, we infer that
\begin{eqnarray*}
	B_{1,\alpha}&\leq& \left( C(X)+o(1)\right)\int_{ B_{\delta}(0) } \frac{1}{\left(1+|Y|^2\right)^{\frac{n-2}{2}(\crit-1)}|Y|^s}\,dY\\
	&\leq&\left( C(X)+o(1)\right)\int_{B_{\delta}(0) } \frac{1}{|Y|^{s}}\,dY\\
	&\leq&\left( C(X)+o(1)\right)\delta^{n-s}.
\end{eqnarray*}
and therefore,
\begin{equation}\label{eq:estb1alpha}
\lim_{\delta\to 0}\lim_{\alpha\to+\infty} B_{1,\alpha}=0.
\end{equation}
\smallskip\noindent We divide $B_{2,\alpha}$ as follows
\begin{equation}\label{eq:Balphad}
B_{2,\alpha}= \int_{R\leq |Y|\leq \frac{\beta_{\alpha}}{\ma}\frac{|X|}{2} }J_{\alpha}\, dY+\int_{\frac{\beta_{\alpha}}{\ma}\frac{|X|}{2} \leq |Y|\leq 2\frac{\beta_{\alpha}}{\ma}|X| }J_{\alpha}\, dY+\int_{|Y|\geq  2\frac{\beta_{\alpha}}{\ma}|X|}J_{\alpha}\, dY.
\end{equation}
Since $|Y|\leq \frac{\beta_{\alpha}}{\ma}\frac{|X|}{2}$, we have that $| X-\frac{\ma}{\beta_\alpha} Y|\geq \frac{|X|}{2}$. Therefore, with the convergence of $\ll_\alpha$ in \eqref{m1} and  $\lim_{\alpha\to +\infty}\frac{\ma}{\beta_{\alpha}}=0$, we get 
\begin{eqnarray}
\int_{R\leq |Y|\leq \frac{\beta_{\alpha}}{\ma}\frac{|X|}{2} }J_{\alpha}\, dY
&\leq& C(X)\, \int_{R\leq |Y|\leq \frac{\beta_{\alpha}}{\ma}\frac{|X|}{2} }\frac{1}{|Y|^{n-s+2}}\, dY\nonumber\\
&\leq&C(X)\left[\left( \frac{\ma}{\beta_{\alpha}} \right)^{2-s}\left( \frac{|X|}{2}\right)^{s-2} -R^{s-2}\right]\nonumber\\
&\leq&  C(X)R^{s-2}.\label{eq:Balphad1}
\end{eqnarray}
For the next term, a change of variable yields
\begin{eqnarray}
\int_{\frac{\beta_{\alpha}}{\ma}\frac{|X|}{2} \leq |Y|\leq 2\frac{\beta_{\alpha}}{\ma}|X| }J_{\alpha}\, dY&\leq& C(X)\, \left( \frac{\ma}{\beta_{\alpha}}\right)^{2-s}\int_{\frac{|X|}{2} \leq |Y|\leq 2|X| }|X-Y|^{2-n}\, dY\nonumber\\
&=&o(1),\label{eq:Balphad2}
\end{eqnarray}
as $\alpha\to +\infty$. Finaly, we estimate the last term. Since  $|Y|\geq  2\frac{\beta_{\alpha}}{\ma}|X|$, we have that $| X-\frac{\ma}{\beta_\alpha} Y|\geq |X|$. Therefore, it follows from the convergence of $\ll_\alpha$ in \eqref{m1} and  $\lim_{\alpha\to +\infty}\frac{\ma}{\beta_{\alpha}}=0$ that
\begin{eqnarray}
\int_{|Y|\geq  2\frac{\beta_{\alpha}}{\ma}|X|}J_{\alpha}\, dY&\leq& C(X)\, \int_{|Y|\geq  2\frac{\beta_{\alpha}}{\ma}|X|}\frac{1}{|Y|^{n-s+2}}\, dY\nonumber\\
&\leq&C(X)\left( \frac{\ma}{\beta_{\alpha}}\right)^{2-s}\nonumber\\
&=&o(1),\label{eq:Balphad3}
\end{eqnarray}
as $\alpha\to +\infty$. Combining \eqref{eq:Balphad}, \eqref{eq:Balphad1}, \eqref{eq:Balphad2} and \eqref{eq:Balphad3}, we infer that 
\begin{equation}\label{eq:estb2alpha}
\lim_{R\to +\infty}\lim_{\alpha\to +\infty} B_{2,\alpha}=0.
\end{equation}
 It follows from \eqref{eq:Balpha12}, \eqref{eq:estb1alpha} and \eqref{eq:estb2alpha}  that the result of this Step. This ends Step \ref{procedure:conv:n=4}.2.2. \qed

\medskip\noindent The equations \eqref{eq:Aalpha+Balpha}, \eqref{eq:Aalpha}, and \eqref{eq:Balpha} yields the result of the Step \ref{procedure:conv:n=4}.2 \qed
\begin{step}\label{conv:n=4}
	We claim that the result of Lemma \ref{eq:Conv:n=4} holds.
\end{step}
\noindent\textit{Proof of Step \ref{conv:n=4}:} For the sake of clarity, we define
$$U_{X,\alpha}:=X^{\beta_1}X^{\beta_2} \partial_i \utza \partial_j \utza,$$
For $R>0$, we write 
\begin{eqnarray*}
	\int_{B_{\delta\ma^{-1}}(0)}U_{X,\alpha} \, dX&=&\int_{B_R(0)} U_{X,\alpha} \, dX+\int_{B_{\delta\ma^{-1}}(0)\backslash B_R(0)} U_{X,\alpha} \, dX.
\end{eqnarray*}
It follows from the strong convergence of \eqref{conv:D12F} that 
\begin{eqnarray}\label{eq:UXalpha}
\ma^2	\int_{B_{\delta\ma^{-1}}(0)}U_{X,\alpha} \, dX&=&\ma^2\int_{B_{\delta\ma^{-1}}(0)\backslash B_R(0)} U_{X,\alpha} \, dX+O(\ma^2).
\end{eqnarray}
We define $\theta_{\alpha}:=\frac{1}{\sqrt{|\ln(\ma) |}}$,  $s_\alpha=\ma^{-\theta_{\alpha}}$ and $t_\alpha=\ma^{\theta_{\alpha}-1}$. We have as $\alpha\to +\infty$ that
\begin{equation}\label{eq:aide1}
\left\{\begin{array}{lll}
s_\alpha=o(t_\alpha);&\ma=o(s_\alpha);&  \ma t_\alpha=o(1)\\
\\
\ln(\frac{1}{t_\alpha\ma})=o(\ln(\frac{1}{\ma}));& \ln(s_\alpha)=o(\ln(\frac{1}{\ma}));&   \ln(\frac{t_\alpha}{s_\alpha})\simeq \ln(\frac{1}{\ma}).
\end{array}\right.
\end{equation}
With \eqref{eq:UXalpha}, we get
\begin{eqnarray}
\ma^2	\int_{B_{\delta\ma^{-1}}(0)}U_{X,\alpha} \, dX&=&\ma^2\left( \int_{B_{\delta\ma^{-1}}(0)\backslash B_{t_{\alpha}}(0)} U_{X,\alpha} \, dX+\int_{B_{t_{\alpha}}(0)\backslash B_{s_{\alpha}}(0)} U_{X,\alpha} \, dX\right. \nonumber\\
&&\left. +\int_{B_{s_{\alpha}}(0)\backslash B_{R}(0)} U_{X,\alpha} \, dX\right) + o\left(\ma^2 \ln\left( \frac{1}{\ma}\right)\right).\label{eq:i0}
\end{eqnarray}
Thanks again to the pointwise control \eqref{ineq:tu} and to \eqref{eq:aide1}, we have 
\begin{equation}\label{eq:i1}
\ma^2	\int_{B_{\delta\ma^{-1}}(0)\backslash B_{t_{\alpha}}(0)} U_{X,\alpha} \, dX
= O\left(\ma^2  \int_{t_\alpha}^{\delta\ma^{-1}}\frac{1}{r}\, dr\right) =O\left(\ma^2\ln\left(\frac{\delta}{t_{\alpha}\ma}\right)\right)=o\left(\ma^2\ln\left(\frac{1}{\ma}\right)\right), 
\end{equation}
and, 
\begin{equation}\label{eq:i2}
\ma^2\int_{B_{s_{\alpha}}(0)\backslash B_{R}(0)} U_{X,\alpha} \, dX=O\left(\ma^2\ln\left(\frac{s_{\alpha}}{R}\right)\right)=o\left(\ma^2\ln\left(\frac{1}{\ma}\right)\right).
\end{equation}
Since $\ma=o(s_{\alpha})$ as $\alpha\to +\infty$, it follows from the result of Step \ref{procedure:conv:n=4} that 
\begin{align*}
	\lim_{\alpha\to +\infty}\sup_{B_{t_{\alpha}}(0)\backslash B_{s_{\alpha}}(0)}\left|   |X|^{2n-2}\partial_i\uta\partial_j\uta-K_n\frac{X_iX_j}{|X|^2}\right|    =0 ,
\end{align*}
where $K_n:=\left( (n-2)K^{n-2}\right)^2$. Therefore, we have that
\begin{eqnarray}
\ma^2\int_{B_{t_{\alpha}}(0)\backslash B_{s_{\alpha}}(0)} U_{X,\alpha} \, dX&=&\ma^2\int_{B_{t_{\alpha}}(0)\backslash B_{s_{\alpha}}(0)} \frac{X^{\beta_1}X^{\beta_2}}{|X|^{2n-2}}\left(|X|^{2n-2}\partial_i \utza \partial_j \utza-K_n\frac{X^iX^j}{|X|^2}\right)  \, dX\nonumber\\
&+&K_n\ma^2\int_{B_{t_{\alpha}}(0)\backslash B_{s_{\alpha}}(0)}  X^{\beta_1}X^{\beta_2}\frac{X^iX^j}{|X|^{2n}} \, dX\nonumber\\
&=& K_n\ma^2\int_{B_{t_{\alpha}}(0)\backslash B_{s_{\alpha}}(0)}  X^{\beta_1}X^{\beta_2}\frac{X^iX^j}{|X|^{2n}} \, dX\nonumber\\
&+&o\left( \ma^2\int_{B_{t_{\alpha}}(0)\backslash B_{s_{\alpha}}(0)} \frac{X^{\beta_1}X^{\beta_2}}{|X|^{2n-2}} \, dX\right)\nonumber\\
&=& K_n\ma^2\int_{B_{t_{\alpha}}(0)\backslash B_{s_{\alpha}}(0)}  X^{\beta_1}X^{\beta_2}\frac{X^iX^j}{|X|^{2n}} \, dX
+o\left( \ma^2\ln(\frac{1}{\ma})\right)\label{eq:o0}.
\end{eqnarray}
On the other hand, since $\ln(\frac{t_\alpha}{s_\alpha})\simeq \ln(\frac{1}{\ma})$, we have 
\begin{eqnarray}
\int_{B_{t_{\alpha}}(0)\backslash B_{s_{\alpha}}(0)}  X^{\beta_1}X^{\beta_2}\frac{X^iX^j}{|X|^{2n}} \, dX&=&\int_{S^{n-1}}\sigma^i \sigma^j \sigma^{\beta_1}\sigma^{\beta_2} \, d\sigma \int_{s_\alpha}^{t_\alpha} \frac{1}{r}\, dr\nonumber\\
&=&\ln\left(\frac{t_\alpha}{s_\alpha}\right)\int_{S^{n-1}}\sigma^i \sigma^j \sigma^{\beta_1}\sigma^{\beta_2} \, d\sigma \nonumber \\
&=&\ln\left(\frac{1}{\ma}\right)\left( 1+o(1)\right) \int_{S^{n-1}}\sigma^i \sigma^j \sigma^{\beta_1}\sigma^{\beta_2} \, d\sigma .\label{eq:o1}
\end{eqnarray}
Combining the equations \eqref{eq:o0}, \eqref{eq:o1}, we obtain that 
\begin{equation}\label{eq:i3}
\ma^2\int_{B_{t_{\alpha}}(0)\backslash B_{s_{\alpha}}(0)} U_{X,\alpha} \, dX=K_n\ma^2\ln\left(\frac{1}{\ma}\right)\left( 1+o(1)\right) \int_{S^{n-1}}\sigma^i \sigma^j \sigma^{\beta_1}\sigma^{\beta_2} \, d\sigma,
\end{equation}
with $K_n:=\left( (n-2)K^{n-2}\right)^2$. The equations \eqref{eq:i0},  \eqref{eq:i1}, \eqref{eq:i2} and \eqref{eq:i3} yields the result of this Lemma.  This ends the proof of Lemma \ref{eq:Conv:n=4}. \qed

\medskip\noindent We are left with estimating $D_\alpha$. Recall that $-\left( \Delta_{\hat{g}}-\Delta_{Eucl}\right)  = \left( \hat{g}^{ij}-\delta^{ij}\right) \partial_{ij}- \hat{g}^{ij}\hat{\Gamma}_{ij}^k\partial_k$ and the Christoffel symbols are $\hat{\Gamma}_{ij}^k:=\frac{1}{2}\hat{g}^{kp}\left(\partial_{i}\hat{g}_{jp}+\partial_{j}\hat{g}_{ip}-\partial_{p}\hat{g}_{ij}\right)$. Then, we write
\begin{equation}\label{def:D}
D_{\alpha}= D_{1,\alpha}-D_{2,\alpha}+\frac{n-2}{2}D_{3,\alpha}-\frac{n-2}{2}D_{4,\alpha},
\end{equation}
where
\begin{align}
	D_{1,\alpha}:=\int_{B_{\delta}(0)}\left( \hat{g}^{ij}-\delta^{ij}\right) X^l \partial_l \uha\partial_{ij}  \uha \, dX &&\bb{, }&\hspace{0.07cm} D_{2,\alpha}:=\int_{B_{\delta}(0)} \hat{g}^{ij}X^l \hat{\Gamma}_{ij}^k\partial_l \uha \partial_k \uha  \, dX,\nonumber\\
	D_{3,\alpha}:=\int_{B_{\delta}(0)}\left( \hat{g}^{ij}-\delta^{ij}\right)\uha\partial_{ij}\uha\, dX &&\bb{, }&\hspace{0.07cm} D_{4,\alpha}:=\int_{B_{\delta}(0)} \hat{g}^{ij} \uha \hat{\Gamma}_{ij}^k\partial_k\uha  \, dX.\label{def:Di}
\end{align}
We now estimate the $D_{i,\alpha}$'s separately. Note that, since the exponential map is normal at $0$, we have that $\partial_{\beta_1}\hat{g}^{ij}(0)=0 $ for all $i,j,\beta_1=1,...,n$. For $i,j,k=1,...,n$, the Taylor formula at $0$ writes
\begin{equation}\label{gamma0}
\Gamma_{ij}^k(X)=\sum_{m=1}^nX^{m}\partial_{m}\Gamma_{ij}^k(0)+O\left(|X|^2 \right),
\end{equation}
and, 
\begin{eqnarray}\label{g0}
\hat{g}^{ij}(X)-\delta^{ij}=\frac{1}{2}\sum_{\beta_1,\beta_2=1}^n X^{\beta_1}X^{\beta_2}\partial_{\beta_1\beta_2}\hat{g}^{ij}(0)+O\left( |X|^3\right). 
\end{eqnarray}

\begin{step}\label{step:55} We claim that
\begin{equation}\label{a01}
\int_{B_{\delta}(0)}|X|^3 |\nabla \uha |^2   \, dX=\left\{\begin{array}{ll} o(\ma^2) &\hbox{ if }n\geq5, \\
O(\ma^2) &\hbox{ if } n= 4,\\
O(\delta\ma)& \hbox{ if } n= 3.
\end{array}\right.
\end{equation}
And,
\begin{eqnarray}\label{a02}
\int_{B_{\delta}(0)} |X|\uha^2  \, dX=\left\{\begin{array}{ll} o(\ma^2) &\hbox{ if }n\geq5, \\
O(\ma^2) &\hbox{ if } n= 4,\\
O(\delta\ma)& \hbox{ if } n= 3.
\end{array}\right.
\end{eqnarray}
\end{step}
\noindent{\it Proof of Step \ref{step:55}:} Estimate \eqref{a02}, this is a direct consequence of the upper bound \eqref{ineq:hu}. We deal with \eqref{a01}. We fix $R>0$ and we write
\begin{eqnarray*}
&&\int_{B_{\delta}(0)}|X|^3 |\nabla \uha |^2   \, dX
=\int_{B_{R\ma}(0)}|X|^3 |\nabla \uha |^2   \, dX+\int_{B_{\delta}(0)\backslash B_{R\ma}(0)}|X|^3 |\nabla \uha |^2   \, dX\\
&&=\ma^{3} \int_{B_{R}(0)}|X|^3 |\nabla \utza|^2   \, dX+\int_{B_{\delta}(0)\backslash B_{R\ma}(0)}|X|^3 |\nabla \uha |^2   \, dX,
\end{eqnarray*}
where $\utza$ is as in \eqref{eq:uta}. It follows from the strong convergence of \eqref{conv:D12F} that $\int_{B_{R}(0)}|X|^3 |\nabla \utza|^2   \, dX=O(1)$ as $\alpha\to +\infty$. As for \eqref{a02}, the control of the integral on $B_\delta(0)\setminus B_{R\ma}(0)$ is a direct consequence of \eqref{ineq:nablatu}. This yields \eqref{a02}. This proves the claim.\qed

\begin{step} We estimate $D_{2,\alpha}$ for $n\geq 4$.
\end{step}
\noindent Since $\hat{g}^{ij}-\delta^{ij}=O(|X|^2)$ as $X\to 0$ and by \eqref{gamma0}, we estimate as $\alpha\to +\infty$ that,
\begin{eqnarray}
D_{2,\alpha}
&=&\delta^{ij}\int_{B_{\delta}(0)} X^l\hat{\Gamma}_{ij}^k\partial_l \uha \partial_k \uha  \, dX+O\left( \int_{B_{\delta}(0)} |X|^3\hat{\Gamma}_{ij}^k\partial_l \uha \partial_k \uha  \, dX\right) \nonumber\\
&=&\sum_{m=1}^n\partial_{m}\hat{\Gamma}_{ii}^k(0)\int_{B_{\delta}(0)}X^lX^{m} \partial_l \uha \partial_k \uha  \, dX+O\left( \int_{B_{\delta}(0)}|X|^3 |\nabla \uha |^2   \, dX\right).\label{c0}
\end{eqnarray}

\noindent The change of variable $Y=\ma^{-1} X$ and the estimates \eqref{c0} and \eqref{a01} yield
\begin{equation}\label{d2alpha0}
	D_{2,\alpha}	=\ma^2\sum_{m=1}^n\partial_{m}\hat{\Gamma}_{ii}^k(0)\int_{B_{\delta\ma^{-1}}(0)}X^lX^{m} \partial_l \utza \partial_k \utza  \, dX+\left\{\begin{array}{ll} o(\ma^2) &\hbox{ if }n\geq5, \\
		O(\ma^2) &\hbox{ if } n= 4.
	\end{array}\right.
\end{equation}
{\bf Case 1:} $n\geq 5$. In this case, $X\mapsto |X|^2\left((1+|X|^2)^{(1-n)/2}\right)^2 \bb{ in } L^1(\rn)$. Therefore, going back to \eqref{d2alpha0}, it follows from the strong convergence \eqref{conv:D12F}, the pointwise convergence of Step \ref{step00}, the estimate control \eqref{ineq:nablatu} and  the Lebesgue dominated convergence theorem that
\begin{eqnarray*}
	D_{2,\alpha}
	&=&\ma^2\sum_{m=1}^n\partial_{m}\hat{\Gamma}_{ii}^k(0)\left(\int_{B_R(0)}X^lX^{m} \partial_l \tilde{u} \partial_k \tilde{u}\, dX+ \int_{\rr^n \backslash B_R(0)}X^lX^{m} \partial_l \tilde{u} \partial_k \tilde{u}\, dX\right) +o\left( \ma^2\right)\\
	&=&\ma^2\sum_{m=1}^n\partial_{m}\hat{\Gamma}_{ii}^k(0)\int_{\rr^n}X^lX^{m} \partial_l \tilde{u} \partial_k \tilde{u}\, dX+o\left( \ma^2\right)\\
\end{eqnarray*}
From the radial symmetry of $\tilde{u}$, we infer that
\begin{eqnarray*}
D_{2,\alpha}	&=&\ma^2\sum_{m=1}^n\partial_{m}\hat{\Gamma}_{ii}^k(0)\int_{\rr^n}X^mX^{k} \left(  \tilde{u}^{\prime}\right)^2 \, dX+o\left( \ma^2\right)\\
&=&\ma^2\sum_{m=1}^n\partial_{m}\hat{\Gamma}_{ii}^k(0)\int_{\rr^n}X^mX^{k} |\nabla \ut|^2 \, dX+o\left( \ma^2\right)\\
&=&\ma^2\sum_{m,k=1}^n\partial_{m}\hat{\Gamma}_{ii}^k(0)\int_{S^{n-1}}\theta^m\theta^ k\, d\theta \int_{0}^{+\infty}r^2 |\nabla_r \ut|^2 \, dr+o\left( \ma^2\right).
\end{eqnarray*}
With the symmetries of the sphere, we have that $ \int_{S^{n-1}}\theta^m\theta^ k\, d\theta =\delta^{mk}\frac{\omega_{n-1}}{n}$. Hence, 
\begin{eqnarray*}
	D_{2,\alpha}	&=&\frac{\ma^2}{n}\omega_{n-1}\sum_{k=1}^n\partial_{k}\hat{\Gamma}_{ii}^k(0)\int_{0}^{+\infty}r^2 |\nabla_r \ut|^2 \, dr+o\left( \ma^2\right)\\
	&=&\frac{\ma^2}{n}\sum_{k=1}^n\partial_{k}\hat{\Gamma}_{ii}^k(0)\int_{\rr^n}|X|^2|\nabla \tilde{u}|^2 \, dX  +o\left( \ma^2\right).
\end{eqnarray*}
{\bf Case 2: $n=4$.} It follows from \eqref{d2alpha0} and the convergence of Lemma \ref{eq:Conv:n=4} that  
\begin{eqnarray*}
	\lim_{\alpha\to +\infty} \frac{1}{\ma^{2}\ln(\frac{1}{\ma})}D_{2,\alpha}&=&4K^4\sum_{m=1}^k \partial_m \hat{\Gamma}_{ii}^k(0)\int_{\sn} (\sigma^l)^2 \sigma^m\sigma^k\, d\sigma \\
	&=&\omega_3K^4\partial_{k}\hat{\Gamma}_{ii}^k(0), \bb{ thanks to }\eqref{eq:4sphere}.
\end{eqnarray*}

\begin{step} We estimate $D_{3,\alpha}$ for $n\geq 4$.
\end{step}
\noindent Thanks to \eqref{g0}, \eqref{a01} and \eqref{a02}, integrations by parts yield 
\begin{eqnarray}
	D_{3,\alpha}&=&\int_{B_{\delta}(0)}\left( \hat{g}^{ij}-\delta^{ij}\right)\uha\partial_{ij}\uha\, dX\nonumber\\
	&=& -\left( \int_{B_{\delta}(0)}\partial_i \hat{g}^{ij}\uha\partial_{j}\uha\, dX +\int_{B_{\delta}(0)}\left( \hat{g}^{ij}-\delta^{ij}\right)\partial_i\uha\partial_{j}\uha\, dX\right)\nonumber \\
	&&+O\left( \int_{\partial B_\delta(0)}|X|^2|\nabla\uha|\uha\, d\sigma\right)  \nonumber\\
	&=& -\frac{1}{2}\left( \int_{B_{\delta}(0)}\partial_i\hat{g}^{ij}\partial_{j}\left( \uha\right)^2 \, dX +\partial_{\beta_1\beta_2}\hat{g}^{ij}(0)\int_{B_{\delta}(0)}X^{\beta_1}X^{\beta_2}\partial_i\uha\partial_{j}\uha\, dX\right)\nonumber\\
	&+& O\left( 	\int_{B_{\delta}(0)}|X|^3 |\nabla \uha |^2   \, dX\right)+O\left( \int_{\partial B_\delta(0)}|X|^2|\nabla\uha|\uha\, d\sigma\right) \nonumber \\
	&=& -\frac{1}{2}\left( -\int_{B_{\delta}(0)}\partial_{ij}\hat{g}^{ij} \uha^2 \, dX +\partial_{\beta_1\beta_2}\hat{g}^{ij}(0)\int_{B_{\delta}(0)}X^{\beta_1}X^{\beta_2}\partial_i\uha\partial_{j}\uha\, dX\right)\nonumber\\
	&&+ O\left( 	\int_{B_{\delta}(0)}|X|^3 |\nabla \uha |^2   \, dX\right)+O\left( \int_{\partial B_\delta(0)}\left(|X|^2|\nabla\uha|\uha+|X|\uha^2\right)\, d\sigma\right) \label{est:d3} 
	\end{eqnarray}
	\begin{eqnarray}
	D_{3,\alpha}&=& -\frac{1}{2}\left( -\int_{B_{\delta}(0)}\partial_{ij}\hat{g}^{ij} \uha^2 \, dX +\partial_{\beta_1\beta_2}\hat{g}^{ij}(0)\int_{B_{\delta}(0)}X^{\beta_1}X^{\beta_2}\partial_i\uha\partial_{j}\uha\, dX\right)\nonumber\\
	&&+\left\{\begin{array}{ll} o(\ma^2) &\hbox{ if }n\geq5, \\
		O(\ma^2) &\hbox{ if } n= 4.
	\end{array}\right.\nonumber\\
	&=& -\frac{1}{2}\left( -\partial_{ij}\hat{g}^{ij}(0) \int_{B_{\delta}(0)}\uha^2 \, dX +\partial_{\beta_1\beta_2}\hat{g}^{ij}(0)\int_{B_{\delta}(0)}X^{\beta_1}X^{\beta_2}\partial_i\uha\partial_{j}\uha\, dX\right)\nonumber\\
	&+& O\left(\int_{B_{\delta}(0)} |X|\uha^2 \, dX \right) +\left\{\begin{array}{ll} o(\ma^2) &\hbox{ if }n\geq5, \\
		O(\ma^2) &\hbox{ if } n= 4.
	\end{array}\right. \nonumber
	\end{eqnarray}
Therefore, with a change variable  $Y=\ma^{-1}X$, we infer that 
\begin{eqnarray}
		D_{3,\alpha}&=&\frac{\ma^2}{2}\left(  \partial_{ij}\hat{g}^{ij}(0) \int_{B_{\delta\ma^{-1}}(0)}\utza^2 \, dX\right.\nonumber  \\
		&-&\left. \partial_{\beta_1\beta_2}\hat{g}^{ij}(0)\int_{B_{\delta\ma^{-1}}(0)}X^{\beta_1}X^{\beta_2}\partial_i\utza\partial_{j}\utza\, dX\right) +\left\{\begin{array}{ll} o(\ma^2) &\hbox{ if }n\geq5, \\
		O(\ma^2) &\hbox{ if } n= 4.
		\end{array}\right. \label{d3alpha}
\end{eqnarray}
{\bf Case 1: $n\geq 5$.} Here again, we have that  $$X\mapsto |X|^2\left((1+|X|^2)^{(1-n)/2}\right)^2\in L^1(\rn)$$ for $n\geq 5$. Therefore, going back to \eqref{d3alpha},  it follows from the strong convergence \eqref{conv:D12F}, the pointwise convergence of Step \ref{step00}, the pointwise control \eqref{ineq:nablatu}, the Lebesgue dominated convergence theorem that   
\begin{eqnarray*}
	D_{3,\alpha}
	&=&\frac{\ma^2}{2}\left(  \partial_{ij}\hat{g}^{ij}(0) \int_{\rr^n}\ut^2 \, dX -\partial_{\beta_1\beta_2}\hat{g}^{ij}(0)\int_{\rr^n}X^{\beta_1}X^{\beta_2}\partial_i\ut\partial_{j}\ut\, dX\right)+o\left( \ma^2\right).
\end{eqnarray*}
{\bf Case 2: $n=4$.} 
Withing again the equations \eqref{convergen4} and  \eqref{convergen40}, we get 
\begin{equation}\label{eq:ud3alpha}
\lim_{\alpha \to +\infty}\frac{1}{\ln(\frac{1}{\ma})}\int_{B_{\delta\ma^{-1}}(0)}\utza^2 \, dX=\omega_3 K^4.
\end{equation}
Using again \eqref{conv:attends:nabla} and \eqref{eq:4sphere}, 
\begin{equation}\label{eq:derived3alpha}
\lim_{\alpha \to +\infty}\frac{\partial_{\beta_1\beta_2}\hat{g}^{ij}(0)}{\ln(\frac{1}{\ma})}\int_{B_{\delta\ma^{-1}}(0)} X^{\beta_1}X^{\beta_2}\partial_i\utza\partial_{j}\utza\, dX=\frac{\omega_3}{6}\left(\partial_{\beta_1\beta_1}\hat{g}^{ii}(0) +2 \partial_{ij}\hat{g}^{ij}(0)\right).
\end{equation}
Then, it follows from \eqref{d3alpha}, \eqref{eq:ud3alpha} and \eqref{eq:derived3alpha} that  
\begin{eqnarray*}
	\lim_{\alpha \to +\infty}\frac{1}{\ma^{2}\ln(\frac{1}{\ma})} D_{3,\alpha}=\frac{\omega_3}{12} K^4\left(  4\partial_{ij}\hat{g}^{ij}(0)-\partial_{\beta_1\beta_1}\hat{g}^{ii}(0)\right). 
\end{eqnarray*}

\begin{step} We estimate $D_{4,\alpha}$ for $n\geq 4$.
\end{step}
Using again integrations by parts, we get 
\begin{eqnarray*}
	D_{4,\alpha}	&=&-\frac{1}{2}\sum_{k=1}^n\int_{B_{\delta}(0)} \partial_k \hat{\Gamma}_{ii}^k \uha^2  \, dX+\frac{1}{2}\int_{\partial B_{\delta}(0)}  \hat{\Gamma}_{ii}^k\uha^2  \vec{\nu}_k\, dX+O\left(\int_{B_{\delta}(0)}  |X| \uha^2  \, dX \right).
\end{eqnarray*}
With \eqref{a02}, we get 
\begin{equation*}
D_{4,\alpha}=-\frac{1}{2}\sum_{k=1}^n\partial_k \hat{\Gamma}_{ii}^k(0)\int_{B_{\delta}(0)}  \uha^2  \, dX+\left\{\begin{array}{ll} o(\ma^2) &\hbox{ if }n\geq5, \\
O(\ma^2) &\hbox{ if } n= 4.
\end{array}\right.
\end{equation*}
\noindent With a change of variable $Y=\ma^{-1}X$, we obtain that
\begin{eqnarray}\label{D4alpha}                       
	D_{4,\alpha}&=&- \frac{\ma^2}{2}\sum_{k=1}^n\partial_k\hat{\Gamma}_{ii}^k(0) \int_{B_{\delta\ma^{-1}}(0)} \utza^2 \, dX+\left\{\begin{array}{ll} o(\ma^2) &\hbox{ if }n\geq5, \\
	O(\ma^2) &\hbox{ if } n= 4.
	\end{array}\right.
\end{eqnarray}
{\bf Case 1: $n\geq 5$.} Here $X\mapsto (1+|X|^2)^{1-n/2}\in L^2(\rn)$. Then with the pointwise convergence of  Step \ref{step00} and the pointwise control \eqref{ineq:tu}, Lebesgue's dominated convergence theorem yields
\begin{eqnarray*}
	D_{4,\alpha}&=&- \frac{\ma^2}{2}\partial_k\hat{\Gamma}_{ii}^k(0) \int_{\rr^n} \tilde{u}^2 \, dX+o(\ma^2).
\end{eqnarray*}
{\bf Case 2: $n=4$.} It follows from \eqref{convergen4} and \eqref{convergen40} that, 
\begin{equation}\label{last0}
\lim_{\alpha\to +\infty}\frac{1}{\ln(\frac{1}{\ma})} \int_{B_{\delta\ma^{-1}}(0)}\utza^2   \, dX =\omega_3 K^4.
\end{equation}
Combining \eqref{D4alpha} and \eqref{last0}, 
\begin{eqnarray*}
	\lim_{\alpha\to +\infty}	\frac{1}{\ma^2\ln(\frac{1}{\ma})}D_{4,\alpha}=- \frac{\omega_3}{2} K^4 \partial_k\hat{\Gamma}_{ii}^k(0).
\end{eqnarray*}

\begin{step} We now deal with $D_{1,\alpha}$ for $n\geq 4$.
\end{step}
\noindent We write 
$$ b^{ijl}=(g^{ij}-\delta^{ij})X^l \bb{ for all } i,j,l=1,...,n.$$
Next, we have that 
	\begin{eqnarray}
D_{1,\alpha}&=&\int_{B_{\delta}(0)}b^{ijl} \partial_l \uha\partial_{ij} \uha\, dX\nonumber\\
&=&-\int_{B_{\delta}(0)}\partial_i b^{ijl} \partial_l \uha\partial_{j} \uha\, dX-\int_{B_{\delta}(0)}b^{ijl}  \partial_j \uha\partial_{il} \uha\, dX\nonumber\\
&+&\int_{\partial B_{\delta}(0)}b^{ijl}  \partial_j \uha\partial_{l} \uha\vec{\nu}_i\, dX.\label{d0}
\end{eqnarray}
	Using the integrations by parts and since $b^{ijl}=b^{jil}$, we get that 
\begin{eqnarray*}
	D_{1,\alpha}^{\prime}&:=&-\int_{B_{\delta}(0)}b^{ijl}  \partial_j \uha\partial_{il} \uha\, dX\\
	&=&\int_{B_{\delta}(0)}b^{ijl} \partial_{lj} \uha\partial_{i} \uha\, dX+\int_{B_{\delta}(0)}\partial_l\left( b^{ijl}\right)  \partial_{j} \uha\partial_{i} \uha\, dX-\int_{\partial B_{\delta}(0)}b^{ijl} \partial_{j} \uha\partial_{i} \uha\vec{\nu_l}\, dX\\
	&=&\int_{B_{\delta}(0)}b^{jil} \partial_{li} \uha\partial_{j} \uha\, dX+\int_{B_{\delta}(0)}\partial_l\left( b^{ijl}\right)  \partial_{j} \uha\partial_{i} \uha\, dX-\int_{\partial B_{\delta}(0)}b^{ijl} \partial_{j} \uha\partial_{i} \uha\vec{\nu_l}\, dX\\
	&=&\int_{B_{\delta}(0)}b^{ijl} \partial_{j} \uha\partial_{il} \uha\, dX+\int_{B_{\delta}(0)}\partial_l\left( b^{ijl}\right)  \partial_{j} \uha\partial_{i} \uha\, dX-\int_{\partial B_{\delta}(0)}b^{ijl} \partial_{j} \uha\partial_{i} \uha\vec{\nu_l}\, dX\\
	&=&-D^{\prime}_{1,\alpha}+\int_{B_{\delta}(0)}\partial_l\left( b^{ijl}\right)  \partial_{j} \uha\partial_{i} \uha\, dX-\int_{\partial B_{\delta}(0)}b^{ijl} \partial_{j} \uha\partial_{i} \uha\vec{\nu}_l\, dX,
\end{eqnarray*}
	then,
\begin{eqnarray}\label{d1}
2 D_{1,\alpha}^{\prime}=\int_{B_{\delta}(0)}\partial_l b^{ijl}  \partial_{j} \uha\partial_{i} \uha\, dX-\int_{\partial B_{\delta}(0)}b^{ijl} \partial_{j} \uha\partial_{i} \uha\vec{\nu}_l\, dX.
\end{eqnarray}
	Combining \eqref{d0} and \eqref{d1}, we get 
\begin{eqnarray}
	D_{1,\alpha}&=&-\int_{B_{\delta}(0)}\partial_i b^{ijl} \partial_l \uha\partial_{j} \uha\, dX+\frac{1}{2}\int_{B_{\delta}(0)}\partial_l b^{ijl}  \partial_{j} \uha\partial_{i} \uha\, dX\nonumber\\
	&+&\int_{\partial B_{\delta}(0)}b^{ijl}  \partial_j \uha\partial_{l} \uha\vec{\nu}_i\, dX-\frac{1}{2}\int_{\partial B_{\delta}(0)}b^{ijl} \partial_{j} \uha\partial_{i} \uha\vec{\nu}_l\, dX\label{exp:D4}
\end{eqnarray}
With \eqref{ineq:nablahu}, we get 
\begin{eqnarray*}
	\int_{\partial B_{\delta}(0)}b^{ijl}  \partial_j \uha\partial_{l} \uha\vec{\nu}_i\, dX&=&O\left( \ma^{n-2}\right) .
\end{eqnarray*} 
Therefore, thanks of \eqref{g0} and \eqref{a01}, we obtain that 
\begin{eqnarray*}
	D_{1,\alpha}&=&-\int_{B_{\delta}(0)}\partial_i b^{ijl} \partial_l \uha\partial_{j} \uha\, dX+\frac{1}{2}\int_{B_{\delta}(0)}\partial_l b^{ijl}  \partial_{j} \uha\partial_{i} \uha\, dX+O\left( \ma^{n-2}\right)\\
	&=&-\int_{B_{\delta}(0)}X^{l}\partial_i \hat{g}^{ij} \partial_l \uha\partial_{j} \uha\, dX-\int_{B_{\delta}(0)}\left( \hat{g}^{ij}-\delta^{ij}\right) \delta^{il} \partial_{l} \uha\partial_{j} \uha\, dX\\
	&&+\frac{1}{2}\int_{B_{\delta}(0)}X^{l}\partial_l \hat{g}^{ij}  \partial_{j} \uha\partial_{i} \uha\, dX+\frac{n}{2}\int_{B_{\delta}(0)}\left( \hat{g}^{ij}-\delta^{ij}\right)   \partial_{j} \uha\partial_{i} \uha\, dX+O\left( \ma^{n-2}\right)\\
	&=&-\partial_{i\beta_1}\hat{g}^{ij}(0)\int_{B_{\delta}(0)}X^{\beta_1}X^{l} \partial_l \uha\partial_{j} \uha\, dX-\frac{1}{2}\partial_{\beta_1\beta_2}\hat{g}^{ij}(0)\int_{B_{\delta}(0)}X^{\beta_1}X^{\beta_2} \delta^{il} \partial_{l} \uha\partial_{j} \uha\, dX\\
	&&+\frac{1}{2}\partial_{l\beta_1}\hat{g}^{ij}(0)\int_{B_{\delta}(0)}X^{\beta_1}X^{l} \partial_{i} \uha\partial_{j} \uha\, dX+\frac{n}{4}\partial_{\beta_1\beta_2}\hat{g}^{ij}(0)\int_{B_{\delta}(0)}X^{\beta_1}X^{\beta_2}  \partial_{i} \uha\partial_{j} \uha\, dX\\
	&&+O\left( 	\int_{B_{\delta}(0)}|X|^3 |\nabla \uha |^2   \, dX\right)  +O\left( \ma^{n-2}\right)\\
	&=&-\partial_{i\beta_1}\hat{g}^{ij}(0)\int_{B_{\delta}(0)}X^{\beta_1}X^{l} \partial_l \uha\partial_{j} \uha\, dX-\frac{1}{2}\partial_{\beta_1\beta_2}\hat{g}^{ij}(0)\int_{B_{\delta}(0)}X^{\beta_1}X^{\beta_2}  \partial_{i} \uha\partial_{j} \uha\, dX\\
	&&+\frac{1}{2}\partial_{l\beta_1}\hat{g}^{ij}(0)\int_{B_{\delta}(0)}X^{\beta_1}X^{l} \partial_{i} \uha\partial_{j} \uha\, dX+\frac{n}{4}\partial_{\beta_1\beta_2}\hat{g}^{ij}(0)\int_{B_{\delta}(0)}X^{\beta_1}X^{\beta_2}  \partial_{i} \uha\partial_{j} \uha\, dX\\
	&&+\left\{\begin{array}{ll} o(\ma^2) &\hbox{ if }n\geq5, \\
		O(\ma^2) &\hbox{ if } n= 4.
	\end{array}\right. 
\end{eqnarray*}
With \eqref{a01}, we observe that
\begin{eqnarray*}
D_{1,\alpha}	&=-\partial_{i\beta_1}\hat{g}^{ij}(0)\int_{B_{\delta}(0)}X^{\beta_1}X^{l} \partial_l \uha\partial_{j} \uha\, dX-\frac{1}{2}\partial_{\beta_1\beta_2}\hat{g}^{ij}(0)\int_{B_{\delta}(0)}X^{\beta_1}X^{\beta_2}  \partial_{i} \uha\partial_{j} \uha\, dX\nonumber\\
&+\frac{1}{2}\partial_{l\beta_1}\hat{g}^{ij}(0)\int_{B_{\delta}(0)}X^{\beta_1}X^{l} \partial_{i} \uha\partial_{j} \uha\, dX \\
&+\frac{n}{4}\partial_{\beta_1\beta_2}\hat{g}^{ij}(0)\int_{B_{\delta}(0)}X^{\beta_1}X^{\beta_2}  \partial_{i} \uha\partial_{j} \uha\, dX+\left\{\begin{array}{ll} o(\ma^2) &\hbox{ if }n\geq5, \\
O(\ma^2) &\hbox{ if } n= 4.
\end{array}\right.\nonumber
\end{eqnarray*}
With the change of variable $Y=\ma^{-1}X$, we get 
\begin{eqnarray}\label{D1alpha0}
	 D_{1,\alpha}&=&\ma^2\left( -\partial_{i\beta_1}\hat{g}^{ij}(0) \int_{B_{\delta\ma^{-1}}(0)}X^{\beta_1} X^l \partial_l \utza\partial_{j} \utza\, dX\right.\nonumber \\
	&& \left. +\frac{n}{4}\partial_{\beta_1\beta_2}\hat{g}^{ij}(0)\int_{B_{\delta\ma^{-1}}(0)}X^{\beta_1}X^{\beta_2}\partial_i\utza\partial_{j}\utza\, dX \right) +\left\{\begin{array}{ll} o(\ma^2) &\hbox{ if }n\geq5, \\
		O(\ma^2) &\hbox{ if } n= 4.
	\end{array}\right.
\end{eqnarray}
\noindent{\bf Case 1: $n\geq 5$.} We have that $X\mapsto   |X|^2\left( 1+|X|^{n-1}\right)^{-2}\in L^{1}(\rr^n)$. Therefore, the strong convergence \eqref{conv:D12F}, the pointwise convergence of Step \ref{step00}, the pointwise control \eqref{ineq:nablatu} and Lebesgue's Convergence Theorem yield
\begin{eqnarray*}
	D_{1,\alpha}&=&\ma^2\left( -\partial_{i\beta_1}\hat{g}^{ij}(0) \int_{\rr^n}X^{\beta_1} X^l \partial_l \ut\partial_{j} \ut\, dX\right. \\
	&& \left. +\frac{n}{4}\partial_{\beta_1\beta_2}\hat{g}^{ij}(0)\int_{\rr^n}X^{\beta_1}X^{\beta_2}\partial_i\ut\partial_{j}\ut\, dX \right) +o\left( \ma^2\right) ,
\end{eqnarray*}
Moreover, since $\tilde{u}$ is a radially symmetrical, we get
\begin{eqnarray*}
	D_{1,\alpha}&=&\ma^2\left( -\partial_{i\beta_1}\hat{g}^{ij}(0) \int_{\rr^n}X^{\beta_1} X^j\left(   \ut^{\prime}\right)^2\, dX\right. \\
	&& \left. +\frac{n}{4}\partial_{\beta_1\beta_2}\hat{g}^{ij}(0)\int_{\rr^n}X^{\beta_1}X^{\beta_2}\partial_i\ut\partial_{j}\ut\, dX \right) +o\left( \ma^2\right)\\
	&=&\ma^2\left( -\partial_{i\beta_1}\hat{g}^{ij}(0)\int_{S^{n-1}}\theta^{\beta_1} \theta^j \, d\theta\int_{0}^{+\infty}r^{n+1}|\nabla_r \ut |^2\, dr\right. \\
	&& \left. +\frac{n}{4}\partial_{\beta_1\beta_2}\hat{g}^{ij}(0)\int_{\rr^n}X^{\beta_1}X^{\beta_2}\partial_i\ut\partial_{j}\ut\, dX \right) +o\left( \ma^2\right)\\
	&=&\ma^2\left( -\frac{1}{n}\omega_{n-1}\partial_{\beta_1 i}\hat{g}^{ij}(0)\delta^{\beta_1 j} \int_{0}^{+\infty}r^{n+1}|\nabla_r \ut |^2\, dr\right. \\
	&& \left. +\frac{n}{4}\partial_{\beta_1\beta_2}\hat{g}^{ij}(0)\int_{\rr^n}X^{\beta_1}X^{\beta_2}\partial_i\ut\partial_{j}\ut\, dX \right) +o\left( \ma^2\right),
\end{eqnarray*}
then, 
\begin{eqnarray*}
	D_{1,\alpha}	&=&\ma^2\left( -\frac{1}{n}\partial_{i j}\hat{g}^{ij}(0)\int_{\rr^n}|X|^2|\nabla \ut |^2\, dX\right. \\
	&& \left. +\frac{n}{4}\partial_{\beta_1\beta_2}\hat{g}^{ij}(0)\int_{\rr^n}X^{\beta_1}X^{\beta_2}\partial_i\ut\partial_{j}\ut\, dX \right) +o\left( \ma^2\right).
\end{eqnarray*}
{\bf Case 2: $n=4$.} It follows from  \eqref{conv:attends:nabla} and \eqref{eq:4sphere} that 
\begin{equation}\label{eq:derived1alpha}
\frac{\partial_{i\beta_1}\hat{g}^{ij}(0)}{\ln(\frac{1}{\ma})} \int_{B_{\delta\ma^{-1}}(0)}X^{\beta_1} X^l \partial_l \utza\partial_{j} \utza\, dX=\omega_3 K^4\partial_{ij}\hat{g}^{ij}(0).
\end{equation}
Therefore, combining \eqref{D1alpha0}, \eqref{eq:derived1alpha} and \eqref{eq:derived3alpha}
\begin{eqnarray*}
	\lim_{\alpha\to +\infty}\frac{\ma^{-2}}{\ln(\frac{1}{\ma})}D_{1,\alpha}=\frac{\omega_3}{6}K^4\left( -4\partial_{ij}\hat{g}^{ij}(0)
	+\partial_{\beta_1\beta_1}\hat{g}^{ii}(0) \right).
\end{eqnarray*}

\begin{step}\label{stepdalpha}
We get as $\alpha \to +\infty$ that,
	\begin{equation*}
	D_{\alpha}=\left\{\begin{array}{cl}  O\left( \delta\ma\right) &\hbox{ if }n=3, \\ \\
	-\ma^2\ln(\frac{1}{\ma})\frac{1}{6}Scal_g(x_0)\omega_3K^4\left( 1+o(1)\right) &\hbox{ if }n=4, \\ \\
-\ma^2  c_{n,s}Scal_g(x_0) \int_{\rr^n}\tilde{u}^2\, dX + o\left( \ma^2\right) &\hbox{ if } n\geq 5.
	\end{array}\right.
	\end{equation*}
	where $c_{n,s}$, $K$ are defined in \eqref{eq:cns}, \eqref{defK}.
\end{step}
\noindent{\it Proof of Step \ref{stepdalpha}:} 
For $n\geq 5$, the steps above yield
\begin{eqnarray*}
	D_{1,\alpha}+\frac{n-2}{2}D_{3,\alpha}&=&\ma^2\left( -\frac{1}{n}\partial_{i j}\hat{g}^{ij}(0)\int_{\rr^n}|X|^2|\nabla \ut |^2\, dX+\frac{n-2}{4}\partial_{ij}\hat{g}^{ij}(0)\int_{\rr^n}  \tilde{u}^2\, dX\right. \\
	&& \left. +\frac{1}{2}\partial_{\beta_1\beta_2}\hat{g}^{ij}(0)\int_{\rr^n}X^{\beta_1}X^{\beta_2}\partial_i\ut\partial_{j}\ut\, dX \right) +o\left( \ma^2\right)\\
	&=&\ma^2\left( -\frac{1}{n}\partial_{i j}\hat{g}^{ij}(0)\int_{\rr^n}|X|^2|\nabla \ut |^2\, dX+\frac{n-2}{4}\partial_{ij}\hat{g}^{ij}(0)\int_{\rr^n}  \tilde{u}^2\, dX\right. \\
	&& \left. +\frac{w_{n-1}^{-1}}{2}\partial_{\beta_1\beta_2}\hat{g}^{ij}(0)\int_{S^{n-1}}\sigma^{i}\sigma^{j}\sigma^{\beta_1}\sigma^{\beta_2} \, d\sigma \int_{\rr^n}|X|^2|\nabla \tilde{u}|^2\, dX \right) +o\left( \ma^2\right).
\end{eqnarray*}
It follows from \cite{B} that
\begin{equation}\label{eq:4sphere}
	\int_{S^{n-1}}\sigma^{i}\sigma^{j}\sigma^{\beta_1}\sigma^{\beta_2} \, d\sigma=\frac{1}{n(n+2)}w_{n-1}\left( \delta^{ij}\delta^{\beta_1\beta_2}+\delta^{i\beta_1}\delta^{j\beta_2}+\delta^{i\beta_2}\delta^{j\beta_1}\right).
\end{equation}
Therefore we get
\begin{eqnarray*}
	D_{1,\alpha}+\frac{n-2}{2}D_{3,\alpha}	&=& \ma^2\left( \frac{1}{n} \left( -\partial_{ij}\hat{g}^{ij}(0)+\frac{1}{2(n+2)}\left( \partial_{\beta_1\beta_1}\hat{g}^{ii}(0)+2\partial_{ij}\hat{g}^{ij}(0)\right) \right) \int_{\rr^n}|X|^2|\nabla \ut |^2\, dX\right. \\
	&+&\left.\frac{n-2}{4}\partial_{ij}\hat{g}^{ij}(0)\int_{\rr^n}  \tilde{u}^2\, dX\right) + o\left( \ma^2\right)\\
	&=& \ma^2\left( \frac{1}{n} \left( -\frac{n+1}{(n+2)}\partial_{ij}\hat{g}^{ij}(0)+\frac{1}{2(n+2)} \partial_{\beta_1\beta_1}\hat{g}^{ii}(0) \right) \int_{\rr^n}|X|^2|\nabla \ut |^2\, dX\right. \\
	&+&\left.\frac{n-2}{4}\partial_{ij}\hat{g}^{ij}(0)\int_{\rr^n}  \tilde{u}^2\, dX\right) + o\left( \ma^2\right).
\end{eqnarray*}
Therefore, using the definition of $D_{\alpha}$, we get 
\begin{eqnarray}
D_{\alpha}&=& D_{1,\alpha}-D_{2,\alpha}+\frac{n-2}{2}D_{3,\alpha}-\frac{n-2}{2}D_{4,\alpha}\nonumber\\
&=& \ma^2\left( \frac{1}{n} \left( -\frac{n+1}{(n+2)}\partial_{ij}\hat{g}^{ij}(0)+\frac{1}{2(n+2)} \partial_{\beta_1\beta_1}\hat{g}^{ii}(0)-\partial_k \hat{\Gamma}_{ii}^k(0)\right) \int_{\rr^n}|X|^2|\nabla \ut |^2\, dX\right. \nonumber\\
&&\quad\quad+\left.\frac{n-2}{4}\left( \partial_{ij}\hat{g}^{ij}(0)+\partial_k \hat{\Gamma}_{ii}^k(0)\right) \int_{\rr^n}  \tilde{u}^2\, dX\right) + o\left( \ma^2\right).\label{Dalpha0}
\end{eqnarray}
Since $\hat{g}^{ij}\hat{g}_{ij}=Id_n$ and $\partial_k \hat{g}^{ij}(0)=0$, we get 
\begin{eqnarray}\label{g1}
\partial_{ij}\hat{g}^{ij}(0)=-\partial_{ij}\hat{g}_{ij}(0) \bb{ for } i,j=1,...,n.
\end{eqnarray}
Combining \eqref{Dalpha0} and \eqref{g1}, we obtain that 
\begin{eqnarray}
	D_{\alpha}&=& \ma^2\left( \frac{1}{n} \left( \frac{n+1}{(n+2)}\partial_{ij}\hat{g}_{ij}(0)-\frac{1}{2(n+2)} \partial_{\beta_1\beta_1}\hat{g}_{ii}(0)-\partial_k \hat{\Gamma}_{ii}^k(0)\right) \int_{\rr^n}|X|^2|\nabla \ut |^2\, dX\right.\nonumber\\
	&&\quad\quad+\left.\frac{n-2}{4}\left( -\partial_{ij}\hat{g}_{ij}(0)+\partial_k \hat{\Gamma}_{ii}^k(0)\right) \int_{\rr^n}  \tilde{u}^2\, dX\right) + o\left( \ma^2\right).\label{a05}
\end{eqnarray}
Thanks again of Jaber \cite{J1}, for $s\in (0,2)$ we have that 
\begin{equation}\label{rap0}
\frac{\int_{\rr^n}|X|^2|\nabla \ut |^2\, dX}{\int_{\rr^n}  \tilde{u}^2\, dX}= \frac{n\left(n-2 \right) \left( n+2-s\right) }{2\left(2n-2-s \right) }.
\end{equation}
On the other hand, Cartan's expansion of the metric $g$ in the exponential chart $\left( B_{\delta}(x_0), \exp_{x_0}^{-1}\right)$ yields
\begin{eqnarray*}
	g_{ij}(x)=\delta_{ij}+\frac{1}{3}R_{ipqj}(x_0)x^px^q+O\left( r^3\right),
\end{eqnarray*}
where $r:=d_g(x,x_0)$.
Since $g$ is $C^{\infty}$, we have that
\begin{eqnarray*}
\partial_{\beta_1\beta_2}g_{ij}(x_0)&=&\frac{1}{3}\left( R_{ipqj}(x_0)\delta_{p\beta_2}\delta_{q\beta_1}+R_{ipqj}(x_0)\delta_{p\beta_1}\delta_{q\beta_2}\right)\nonumber\\
&=&\frac{1}{3}\left( R_{i\beta_2\beta_1j}(x_0)+R_{i\beta_1\beta_2j}(x_0)\right).
\end{eqnarray*}
The Bianchi identities and the symmetry yields $R_{iijj}=0$ and $R_{ij\alpha\beta}=-R_{ij\beta\alpha}$. Since $R_{ijij}=Scal_g(x_0)$, we then get that
 
\begin{equation}\label{g2}
\sum_{i,j=1}^n	\partial_{ij}g_{ij}(x_0)= \frac{1}{3}Scal_g(x_0) \bb{ and } 	\sum_{i,\beta_1=1}^n	\partial_{\beta_1\beta_1}g_{ii}(x_0)= -\frac{2}{3}Scal_g(x_0).
\end{equation}
Now, using the Christoffel symbols and $\partial_k g^{ij}(0)=0$, we obtain that  
\begin{eqnarray*}
	\partial_k\Gamma_{ii}^k(x_0)&=&\frac{1}{2}\left( \partial_{ki} g_{ik}+\partial_{ki}g_{ik}-\partial_{kk}g_{ii}\right)\left( x_0\right)\\
	&=&\frac{1}{6}\left( R_{iikk}+R_{ikik}-R_{iikk}+R_{ikik}-2R_{ikki}\right),
\end{eqnarray*}
then we have
\begin{equation}\label{g3}
\sum_{i,k=1}^n\partial_k\Gamma_{ij}^k(x_0)= \frac{2}{3}Scal_g(x_0).
\end{equation}
Combining \eqref{a05}, \eqref{rap0}, \eqref{g2} and \eqref{g3}, we get that
\begin{eqnarray*}
	D_{\alpha}&=& \ma^2\left( \left( \frac{n+1}{n+2}\partial_{ij}\hat{g}_{ij}(0)-\frac{1}{2(n+2)} \partial_{\beta_1\beta_1}\hat{g}_{ii}(0)-\partial_k \hat{\Gamma}_{ij}^k(0)\right) \frac{\left(n-2 \right) \left( n+2-s\right) }{2\left(2n-2-s \right) }\right.\\
	&&+\left.\frac{n-2}{4}\left( -\partial_{ij}\hat{g}_{ij}(0)+\partial_k \hat{\Gamma}_{ii}^k(0)\right)   \right)\int_{\rr^n}\tilde{u}^2\, dX + o\left( \ma^2\right)\\
	&=&\ma^2\left( \left( \frac{n+1}{3(n+2)}+\frac{1}{3(n+2)} -\frac{2}{3}\right)Scal_g(x_0) \frac{\left(n-2 \right) \left( n+2-s\right) }{2\left(2n-2-s \right) }\right.\\
	&&+\left.\frac{n-2}{4}Scal_g(x_0) \left( -\frac{1}{3}+\frac{2}{3}\right)   \right)\int_{\rr^n}\tilde{u}^2\, dX + o\left( \ma^2\right)\\
	&=&\ma^2Scal_g(x_0) \left( -\frac{\left(n-2 \right) \left( n+2-s\right) }{6\left(2n-2-s \right) }+\frac{n-2}{12}  \right)\int_{\rr^n}\tilde{u}^2\, dX + o\left( \ma^2\right)\\
\end{eqnarray*}
This ends Step \ref{stepdalpha} for $n\geq 5$. The analysis is similar when $n=4$.\qed

\begin{step} We prove Theorem \ref{th1} for $n\geq4$.
\end{step}
First,	using the definitions \eqref{bcd} of $B_{\alpha}$, $C_{\alpha}$ and $D_{\alpha}$ and thanks to Steps \ref{stepcalpha} to \ref{stepdalpha}, we get
 \begin{equation}\label{Aalpha}
C_{\alpha}+D_\alpha=\left\{\begin{array}{cl} \ma^2\left(a_{\infty}(x_0)-c(n,s) Scal_g(x_0) \right) \int_{\rr^n}\tilde{u}^2\, dX+o(\ma^2) &\hbox{ if }n\geq5, \\ \\
\ma^2 \ln (\frac{1}{\ma })\left( \left(a_{\infty}(x_0)-\frac{1}{6} Scal_g(x_0) \right)\omega_3 K^4+o(1)\right)  &\hbox{ if } n= 4,\\ \\
O(\delta\ma)&\hbox{ if } n= 3.
\end{array}\right.
\end{equation}
We distinguish three cases:

\medskip\noindent {\bf Case 1:} If $n\geq 5$, \eqref{balpha} and \eqref{Aalpha} yield
$$\left(a_{\infty}(x_0)-c(n,s) Scal_g(x_0) \right) \int_{\rr^n}\tilde{u}^2\, dX=O\left(\ma^{n-4}\right)=o(1) $$
and then $a_{\infty}(x_0)=c(n,s) Scal_g(x_0)$, with $c(n,s)$ as in \eqref{eq:cns}.

\smallskip\noindent {\bf Case 2:} If $n=4$, the proof is similar.

\begin{step}\label{step:n3} We prove Theorem \ref{th1} when $n=3$.
\end{step}

\smallskip\noindent{\it Step \ref{step:n3}.1:} We claim that
\begin{equation}\label{claim:n3}
C_\alpha+D_\alpha=O(\delta\ma)\hbox{ as }\alpha\to +\infty.
\end{equation}
We prove the claim. It follows from \eqref{calpha} that 
$$C_\alpha=O\left(\int_{B_\delta(0)}\uha^2\, dx+\int_{\partial B_\delta(0)}|X|\uha^2\, d\sigma\right)$$
as $\alpha\to +\infty$. The definitions \eqref{def:Di} of $D_{i,\alpha}$, $i=2,4$ yield
$$D_{2,\alpha}=O\left(\int_{B_\delta(0)}|X|^2|\nabla \uha|^2\, dx\right)\hbox{ and }D_{4,\alpha}=O\left(\int_{B_\delta(0)}|X|\cdot |\nabla \uha|\uha\, dx\right).$$ 
The identity \eqref{exp:D4} yields
$$D_{1,\alpha}=O\left(\int_{B_\delta(0)}|X|^2|\nabla \uha|^2\, dx+\int_{\partial B_\delta(0)}|X|^3|\nabla\uha|^2\, d\sigma\right).$$
It follows from \eqref{est:d3} that
$$D_{3,\alpha}=O\left(\int_{B_\delta(0)}(\uha^2+|X|^2|\nabla \uha|^2)\, dx+\int_{\partial B_\delta(0)}(|X|^3|\nabla\uha|^2+|X|\uha^2)\, d\sigma\right).$$
Therefore, with \eqref{def:D}, we get that
$$C_\alpha+D_{\alpha}=O\left(\int_{B_\delta(0)}(\uha^2+|X|^2|\nabla \uha|^2)\, dx+\int_{\partial B_\delta(0)}(|X|^3|\nabla\uha|^2+|X|\uha^2)\, d\sigma\right).$$
It then follows from \eqref{est:co} and \eqref{m41} that
$$C_\alpha+D_{\alpha}=O\left(\ma\int_{B_\delta(0)} |X|^{-2}\, dx+\ma\int_{\partial B_\delta(0)} |X|^{-1}\, d\sigma\right)=O(\delta\ma)$$
since $n=3$. This proves \eqref{claim:n3}.

\smallskip\noindent{\it Step \ref{step:n3}.2:} We write the Green's function as in \eqref{green0} with $\beta_{x_0}\in C^{2}(M\backslash \{x_0\})\cap C^{0,\theta}(M)$ where $\theta \in (0, 1)$. 
In particular, 
\begin{equation}\label{greenfunction}
\hat{G}_{x_0}(x):=G(x_0,\exp_{x_0}(X))=\frac{1}{4\pi |X|}+\beta_{x_0}(\exp_{x_0}(X)) \bb{ for all } x\in B_{\delta}(0).
\end{equation}
Combining \eqref{balpha} and \eqref{Aalpha}, we get that 
\begin{equation}\label{m51}
d_3^2 \int_{\partial B_{\delta}(0)} \delta \left(\frac{|\nabla \hat{G}_{x_0} |^2}{2}  +\hat{a}_{\infty}\frac{\hat{G}_{x_0}^2}{2}\right) -\frac{1}{\delta}\left( \langle X, \nabla \hat{G}_{x_0}\rangle^2 +\frac{1}{2}\langle X, \nabla \hat{G}_{x_0}\rangle \hat{G}_{x_0} \right) \, d\sigma=O(\delta)
\end{equation}
From \eqref{greenfunction}, we denote that:
\begin{eqnarray*}
&&|\nabla \hat{G}_{x_0} |^2=\frac{1}{16\pi^2\delta^4}+|\nabla \beta_{x_0} |^2-\frac{1}{2\pi \delta^{3}}\langle X,\nabla\beta_{x_0}\rangle, \\
&&\hat{G}_{x_0}^2=\frac{1}{16\pi^2\delta^2}+\beta_{x_0}^2+\frac{1}{2\pi \delta}\langle X,\beta_{x_0}\rangle,\\
  &&\langle X, \nabla \hat{G}_{x_0}\rangle^2=\frac{1}{16\pi^2\delta^2}+\langle X,\nabla\beta_{x_0}\rangle^2-\frac{1}{2\pi \delta}\langle X,\nabla\beta_{x_0}\rangle, \\
&& \langle X, \nabla \hat{G}_{x_0}\rangle \hat{G}_{x_0} =-\frac{1}{16\pi^2\delta^2}-\frac{1}{4\pi\delta}\beta_{x_0}+\frac{\langle X,\nabla\beta_{x_0}\rangle}{4\pi\delta}+\langle X,\nabla\beta_{x_0}\rangle\beta_{x_0}.
\end{eqnarray*}
We replace all the terms in \eqref{m51} and get 
\begin{eqnarray*}
d_n^2 \int_{\partial B_{\delta}(0)} \delta \left(\frac{|\nabla \beta_{x_0} |^2}{2}  +\hat{a}_{\infty}\frac{\beta_{x_0}^2}{2}\right)-\frac{1}{4\pi \delta^{2}}\langle X,\nabla\beta_{x_0}\rangle +\hat{a}_{\infty} \left( \frac{1}{16\pi^2\delta}+\frac{1}{2\pi}\langle X,\beta_{x_0}\rangle\right) \\
-\frac{1}{\delta}\left( \langle X,\nabla\beta_{x_0}\rangle^2 -\frac{1}{2\pi \delta}\langle X,\nabla\beta_{x_0}\rangle+\frac{1}{2}\left( \frac{-1}{4\pi\delta}\beta_{x_0} +\frac{\langle X,\nabla\beta_{x_0}\rangle}{4\pi\delta}+\langle X,\nabla\beta_{x_0}\rangle\beta_{x_0}\right) \right) \, d\sigma=O(\delta).
\end{eqnarray*}
We note that,
\begin{eqnarray*}
\lim_{\delta\to 0} \sup_{X\in \partial B_{\delta}(0)}\langle X,\nabla\beta_{x_0}\rangle=0.
\end{eqnarray*}
We multiply the last equation by $\delta^2$ and passing the limit $\delta\to 0$, we get that $\beta_{x_0}(x_0):=\beta_{x_0}(\exp_{x_0}(0))=0$, so the mass vanishes at $x_0$.\qed
\section{Proof of Theorem \ref{theo:B_s(g)}}\label{sec:theo:B_s(g)}

 We assume that that there is no extremal of \eqref{ineq:opt}, i.e. for all $u\in H_1^2(M)\backslash\{0\}$, we have that 
\begin{equation}\label{eq:nonextr}
\vv u \vv_{\crit}^{2}< \mu_{s}(\rr^n)^{-1}\left( \int_M |\nabla u|^2 \, dv_g+ B_s(g)\int_M u^2 \, dv_g\right). 
\end{equation} 
We define $a_{\alpha}(x):= B_s(g)-\frac{1}{\alpha}>0$ for all  $x\in M$ and $\alpha>0$ large. We define the functional  
$$J_{\alpha}(u)=\frac{\int_M\left(  |\nabla u|^2 dv_g+a_{\alpha}u^2\right) \, dv_g}{\left( \int_M\frac{u^{\crit}}{d_g(x,x_0)^s}\, dv_g\right)^{\frac{2}{\crit}}} \bb{ for } u\in H^{2}_1(M)\backslash \{0\}.$$
It then follows from the definition of $B_s(g)$ that there exists $w\in H^{2}_1(M)\backslash \{0\}$ such that $J_{\alpha}(w)< \mu_{s}(\rr^n)$, and therefore
\begin{equation}\label{ext:0}
\inf_{u\in N_s(M)} J_{\alpha}(u)< \mu_{s}(\rr^n),
\end{equation} 
where $$ N_s(M):=\{ u\in H_1^{2}(M) , \vv u\vv_{\crit}=1\}.$$
Set $$\ll_{\alpha}:= \inf_{u\in N_s(M)} J_{\alpha}(u).$$
By the assumption \eqref{ext:0}, classical arguments (see Jaber \cite{J1}) yield the existence of a non negative minimizer $u_{\alpha}\in N_s(M)$ for $\ll_{\alpha}$. The Euler-Lagrange's equation for $u_{\alpha}$ is then
\begin{equation}\label{ext:1}
\Delta_g u_{\alpha}+a_{\alpha} u_{\alpha}=\ll_{\alpha} \frac{u_{\alpha}^{\crit-1}}{d_g(x,x_0)^s}\hbox{ in }H_1^2(M).
\end{equation}
It follows from the regularity and the maximum principle of Jaber \cite{J1} that $u_{\alpha}\in C^{0,\beta_1}(M)\cap C_{loc}^{2,\beta_2}(M\backslash\{x_0\})$, $\beta_1\in (0,\min(1,2-s))$, $\beta_2\in (0,1)$ and $u_{\alpha}>0$.
\begin{step}\label{step_u_0}
	We claim that,
	$$u_{\alpha}\rightharpoonup 0 \bb{ weakly in } H_{1}^2(M) \bb{ as } \alpha\to +\infty.$$
\end{step}
\noindent{\textit{Proof of Step \ref{step_u_0}:}} For any $\alpha>0$, we have $\vv u_{\alpha}\vv_{\crit}=1$ and $J_{\alpha}(u_{\alpha})=\ll_{\alpha}<\mu_{s}(\rr^n)$, and we get $(u_{\alpha})_{\alpha>0}$ is bounded in $H_1^2(M)$. Then, there exists $u_0\in H_1^2(M)$ such that $u_{\alpha}\rightharpoonup u_0$ in $H_{1}^2(M)$ as $\alpha\to +\infty$. If $u_0\not\equiv 0$, taking the limit in equation \eqref{ext:1}, we get 
\begin{equation}\label{eq:equal}
\Delta_g u_0+ B_s(g) u_0=\ll \frac{u_0^{\crit-1}}{d_g(x,x_0)^s},
\end{equation}
where $\lambda:=\lim_{\alpha\to +\infty}\lambda_\alpha$ (up to extraction). It follows from \eqref{eq:nonextr} and \eqref{eq:equal} that
\begin{eqnarray*}
	\mu_{s}(\rr^n)< \frac{\left( \int_M |\nabla u_0|^2 \, dv_g+ B_s(g)\int_M u_0^2 \, dv_g\right)}{\vv u_0 \vv_{\crit}^{2}}= \ll  \left(\int_M \frac{u_0^{\crit}}{d_g(x,x_0)} \right)^{1-\frac{2}{\crit}} 
\end{eqnarray*}
Since $\ll \leq \mu_{s}(\rr^n)$ and
$$\int_M \frac{u_0^{\crit}}{d_g(x,x_0)} \, dv_g \leq \lim_{\alpha\to +\infty} \inf \int_M \frac{u_{\alpha}^{\crit}}{d_g(x,x_0)} \, dv_g=1.$$
We get that, $\ll = \mu_{s}(\rr^n)$.
Therefore, $u_0$ is a nonzero extremal function of \eqref{eq:nonextr} contradiction. Hence $u_0\equiv 0$.

\qed
\begin{step}\label{step_lamda}
	We claim that, 
	$$\ll_{\alpha}\to \mu_{s}(\rr^n)\bb{ as } \alpha \to +\infty.$$
\end{step}
\noindent{\textit{Proof of Step \ref{step_lamda}:}}
Since for all $\alpha>0$, we have $0<\ll_{\alpha} <\mu_{s}(\rr^n)$ then, up to a subsequence, $\ll_{\alpha}\to \ll \leq \mu_{s}(\rr^n)$ as $\alpha\to +\infty$. We proceed by contradiction and assume that $\ll\neq \mu_{s}(\rr^n)$. Then there exists $\ep_0$ and $\alpha_0>0$ such that for all $\alpha>\alpha_0$,
$$\mu_{s}(\rr^n)>\ll+\ep_0.$$
Thanks of Jaber \cite{J1}, there exists $B_1$ such that for all $\alpha>0$, we have 

\begin{equation*}
\left(\int_M \frac{|\ua|^{\crit}}{d_g(x,x_0)^s}\, dv_g\right)^{\frac{2}{\crit}}\leq \mu_s(\rr^n)^{-1}\int_M|\nabla \ua|_g^2\, dv_g+B_1\int_M\ua^2\, dv_g.
\end{equation*}
By the last Step and since the embedding of $H_1^2(M)$ in $L^2(M)$ is compact, 
$$u_{\alpha}\to 0 \bb{ in } L^2(M) \bb{ as } \alpha \to +\infty.$$
Therefore, $\vv u_{\alpha}\vv_{\crit}=1$ and $J_{\alpha}(u_{\alpha})=\ll_{\alpha}$, we have

$$1\leq \frac{\ll_\alpha}{\ll+\ep_0} +o(1).$$
Letting $\alpha\to +\infty$ in the last relation, we obtain that $\frac{\ll}{\ll+\ep_0}\geq 1$, a contradiction since $\ll\geq 0$ and $\ep_0>0$. \qed

\medskip \noindent We are in position to prove Theorem \ref{th1}. Since $u_\alpha$ above satisfies the hypothesis of Theorem \ref{th1}, we have that $B_s(g)=c_{n,s} Scal_g(x_0)$ if $n> 4$ and $m_{B_s(g)}(x_0)=0$ if $n=3$.  \qed
	\section{Appendix}
These results and their proofs are closely to the work of Jaber \cite{J2}. We fix $\delta_0 \in (0, i_g(M))$ where $i_g(M)>0$ is the injectivity radius of $(M,g)$. We fix $\eta_0\in C^{\infty}(B_{\frac{3\delta_0}{4}}(0)\subset \rr^n)$ such that $\eta\equiv 1$ in $B_{\frac{\delta_0}{2}}(0)$.
\begin{theo}\label{theo1}
	We let $\left( u_{\alpha}\right) _{\alpha>0}$ be as in \eqref{m2}. We consider a sequence $(z_{\alpha})_{\alpha>0}\in M$ such that $\lim_{\alpha\to +\infty} z_{\alpha}=x_0.$
	We define the function 
	\begin{eqnarray*}
	\tilde{u}_{\alpha}(X):=\mu_{\alpha}^{\frac{n-2}{2}} u_{\alpha}(\exp_{z_{\alpha}}(\mu_{\alpha}X))\bb{ for all } X\in B_{\mu_{\alpha}^{-1}\delta_0}(0)\subset \rr^n,
	\end{eqnarray*}
 where $\exp_{z_{\alpha}}:B_{\delta_0}(0)\to B_{\delta_0}(z_{\alpha})\subset M$ is the exponential map at $z_{\alpha}$. We assume that 
	$$ d_g(x_{\alpha},z_{\alpha})=O(\ma) \bb{ when } \alpha\to +\infty.$$
	Then, $$ d_g(z_{\alpha},x_0)=O(\ma) \bb{ when } \alpha\to +\infty,$$
	and, up to a subsequence, $\eta_{\alpha}\tilde{u}_{\alpha}\to \tilde{u}$ weakly in $D^{2}_1(\rr^n)$ (the completion of $C^\infty_c(\rn)$ for $\Vert\nabla\cdot\Vert_2)$ and uniformly in $C^{0,\beta}_{loc}(\rr^n)$, for all $\beta \in \left( 0,\min\{1,2-s\}\right)$, where $\eta_{\alpha}:=\eta_0(\ma \cdot)$ and 
	$$  \tilde{u}(X)=\left(\frac{\left( c_0^{2-s}(n-2)(n-s)\mu_s(\rn)^{-1}\right)^{\frac{1}{2}} }{c_0^{2-s}+|X-X_0|^{2-s}} \right)^{\frac{n-2}{2-s}} \bb{ for all } X\in \rr^n, $$
	with $X_0\in \rr^n$, $c_0>0$. In particular, $\tilde{u}$ satisfies 
	\begin{eqnarray}\label{eq:norme1}
	\Delta_{Eucl} \tilde{u}= \mu_s(\rn)\frac{\tilde{u}^{\crit-1}}{|X-X_0|^s} \bb{ in } \rr^n \bb{ and } \int_{\rr^n}\frac{\tilde{u}^{\crit}}{|X-X_0|^s} \, dX=1,
	\end{eqnarray}
	where $Eucl$ is the Euclidean metric of $\rr^n$.\end{theo}
	\begin{proof}
We define the metric $\bar{g}_{\alpha}(X):=\left( \exp^\star_{z_{\alpha}} g\right) (\mu_{\alpha}X)$ in $\rn$ and we consider the vector $ X_{0,\alpha}=\ma ^{-1} \exp_{z_{\alpha}}^{-1} (x_0)$.
Since $u_{\alpha}$ verifies the equation \eqref{m2}, we get $\tilde{u}_{\alpha}$ verifies also weakly  
$$\Delta_{\bar{g}_{\alpha}}\tilde{u}_{\alpha}+\tilde{a}_{\alpha} \tilde{u}_{\alpha}=\lambda_{\alpha}\frac{ \tilde{u}_{\alpha}^{\crit-1}}{d_{\bar{g}_{\alpha}}(X,X_{0,\alpha})^s} \bb{ in }\rr^n,$$
where $\tilde{a}_{\alpha}(X):=\ma^2 a_{\alpha}(\exp_{z_{\alpha}}(\ma X))\to 0 $ as $\alpha\to +\infty$.	
Next, we follow the same proof of Theorem 2 in Jaber \cite{J2} and we get Theorem \ref{theo1}.\end{proof}

\smallskip\noindent{\bf Acknowledgement:} The author thanks Fr\'ed\'eric Robert for his advice, encouragements, comments and careful reading of this work.
\bibliographystyle{abbrv}

\end{document}